\documentclass{article}
\usepackage{graphicx} 
\usepackage{amsthm}
\usepackage{amsfonts, amsmath, amssymb, amsthm}
\usepackage{cleveref}
\usepackage{url}
\usepackage[utf8x]{inputenc}
\usepackage{cancel}

\author{Quentin Le Hou\'erou \and Ludovic Levy Patey \and Keita Yokoyama}

\title{Conservation of Ramsey's theorem for pairs\\ and well-foundedness}
\date{\today}

\newtheorem{theorem}{Theorem}
\numberwithin{theorem}{section}
\newtheorem{maintheorem}[theorem]{Main Theorem}
\newtheorem{lemma}[theorem]{Lemma}
\newtheorem{question}[theorem]{Question}
\newtheorem{proposition}[theorem]{Proposition}
\newtheorem{statement}[theorem]{Statement}
\newtheorem{remark}[theorem]{Remark}
\newtheorem{definition}[theorem]{Definition}
\newtheorem{corollary}[theorem]{Corollary}

\makeatletter
\newtheorem*{rep@theorem}{\rep@title}
\newcommand{\newreptheorem}[2]{%
\newenvironment{rep#1}[1]{%
 \def\rep@title{#2 \ref{##1}}%
 \begin{rep@theorem}}%
 {\end{rep@theorem}}}
\makeatother

\newreptheorem{maintheorem}{Main Theorem}

\usepackage{xcolor}

\newcommand{\RCA}{\mathsf{RCA}}
\newcommand{\WKL}{\mathsf{WKL}}
\newcommand{\SRT}{\mathsf{SRT}}
\newcommand{\EM}{\mathsf{EM}}
\newcommand{\SEM}{\mathsf{SEM}}
\newcommand{\ADS}{\mathsf{ADS}}

\newcommand{\SADS}{\mathsf{SADS}}
\newcommand{\COH}{\mathsf{COH}}
\newcommand{\ACA}{\mathsf{ACA}}

\newcommand{\qvdash}{\operatorname{{?}{\vdash}}}
\newcommand{\nqvdash}{\operatorname{{?}{\nvdash}}}

\newcommand{\BSig}{\mathsf{B}\Sigma^0}
\newcommand{\ISig}{\mathsf{I}\Sigma^0}

\newcommand{\IDel}{\mathsf{I}\Delta^0}
\newcommand{\BCDe}{\mathsf{BC}\Delta^0}
\newcommand{\WF}{\mathsf{WF}}
\newcommand{\RT}{\mathsf{RT}}
\newcommand{\BME}{\mathsf{BME}}

\renewcommand{\L}{\mathcal{L}}
\newcommand{\C}{\mathcal{C}}
\newcommand{\R}{\mathcal{R}}
\renewcommand{\S}{\mathcal{S}}
\newcommand{\T}{\mathcal{T}}
\renewcommand{\P}{\mathcal{P}}
\newcommand{\M}{\mathcal{M}}
\newcommand{\Nc}{\mathcal{N}}
\renewcommand{\P}{\mathcal{P}}

\newcommand{\NN}{\mathbb{N}}

\newcommand{\uh}[0]{{\upharpoonright}}

\newcommand{\card}{\operatorname{card}}

\makeatletter
\DeclareFontEncoding{LS2}{}{\@noaccents}
\makeatother
\DeclareFontSubstitution{LS2}{stix}{m}{n}
\DeclareSymbolFont{arrows3}{LS2}{stixtt}{m}{n}
\DeclareMathSymbol{\smashtimes}{\mathbin}{arrows3}{"A4}

\usepackage{mathabx}
\newcommand{\nplus}{\dotplus}
\newcommand{\ntimes}{\dottimes}

\newcommand{\nsum}{\dot{\sum}}

\def\qt#1{``#1''}%

\begin{document}

\maketitle

\begin{abstract}
In this article, we prove that Ramsey's theorem for pairs and two colors is  $\Pi^1_1$-conservative over~$\RCA_0 + \BSig_2 + \WF(\epsilon_0)$ and   over~$\RCA_0 + \BSig_2 + \bigcup_n \WF(\omega^\omega_n)$. These results improve theorems from Chong, Slaman and Yang~\cite{chong2017inductive} and Ko{\l}odziejczyk and Yokoyama~\cite{kolo2021search} and belong to a long line of research towards the characterization of the first-order part of Ramsey's theorem for pairs. 
\end{abstract}

\section{Introduction}

We investigate the proof-theoretic strength of Ramsey's theorem for pairs and some of its consequences, by proving that they are $\Pi^1_1$-conservative over the base theory of Reverse Mathematics ($\RCA_0$), augmented with the collection principle for $\Sigma^0_2$ formulas ($\BSig_2$) and the assumption that some appropriate ordinals are well-founded ($\WF(\alpha)$). This article belongs to a long line of research towards the characterization of the first-order part of Ramsey's theorem for pairs, and is part of a larger foundational program known as Reverse Mathematics, and whose goal is to find optimal axioms for proving ordinary theorems. See Simpson~\cite{simpson_2009} or Dzhafarov and Mummert~\cite{dzhafarov2022reverse} for a reference book on Reverse Mathematics, and Hirschfeldt~\cite{hirschfeldt2017slicing} to a gentle introduction to the reverse mathematics of combinatorial principles.

Ramsey's theorem is the founding theorem of Ramsey's theory, a field of research which studies the conditions under which structure appears among a sufficiently large amount of data. We identify a set~$k$ with its ordinal $\{0, 1, \dots, k-1\}$.  Given a set~$X \subseteq \NN$ and some integer~$n \in \NN$, we let $[X]^n$ denote the set of all unordered $n$-tuples over~$X$. 
Given a coloring $f : [\NN]^n \to k$, a set~$H \subseteq \NN$ is \emph{$f$-homogeneous (for color~$i < k$)}  if $f(\sigma) = i$ for every~$\sigma \in [\NN]^n$. 

\begin{statement}[Ramsey's theorem]
Given $n, k \in \NN$, $\RT^n_k$ is the statement \qt{Every coloring $f : [\NN]^n \to k$ admits an infinite $f$-homogeneous set}.
\end{statement}

The proof-theoretic strength of Ramsey's theorem for triples and larger size of tuples is well-understood: $\RT^n_k$ is equivalent to the Arithmetic Comprehension Axiom ($\ACA_0$) over~$\RCA_0$ whenever $n \geq 3$ and $k \geq 2$ (see Jockusch~\cite{jockusch1972ramsey} for a computability-theoretic equivalence, and Simpson~\cite[Theorem III.7.6]{simpson_2009} for a proof-theoretic version of it). The first-order part of~$\ACA_0$ is precisely Peano Arithmetic.

Ramsey's theorem for pairs and two colors ($\RT^2_2$) can be phrased as stating, for every infinite graph, the existence of an infinite subset of vertices such that the induced subgraph is either a clique or an anticlique. $\RT^2_2$ is known to be strictly weaker than $\ACA_0$ over~$\RCA_0$ (see Seetapun and Slaman~\cite{seetapun1995strength}), and incomparable with weak K\"onig's lemma ($\WKL_0$), the restriction of K\"onig's lemma to binary trees (see Liu~\cite{liu2012rt22}).
The characterization of the first-order part of Ramsey's theorem for pairs and two colors is a long-standing open question.

\begin{definition}[Collection principle]
Given a family $\Gamma$ of formulas, $\mathsf{B}\Gamma$ is the statement \qt{For every formula~$\varphi \in \Gamma$, and every~$a \in \NN$, if $\forall x < a \exists y \varphi(x, y)$, then there is some~$b \in \NN$ such that $\forall x < a \exists y < b \varphi(x, y)$}.
\end{definition}

Ramsey's theorem for pairs implies $\BSig_2$ over~$\RCA_0$ (see Hirst~\cite{hirst1987combinatorics}). The following question is a major open problem in Reverse Mathematics:  

\begin{question}\label[question]{quest:main-question}
Is $\RT^2_2$ a $\Pi^1_1$-conservative extension of~$\RCA_0 + \BSig_2$?
\end{question}

\subsection{Well-foundedness and bounded monotone enumerations}

In Computability Theory, the ability to compute fast-growing functions can be seen as a unified framework to measure the strength of some Turing degrees, e.g., hyperimmune degrees, high degrees, and degrees admitting a modulus. In Proof Theory, the strength of a theory is better measured by the ability to prove the well-foundedness of some ordinals. The \emph{$\Pi^1_1$ proof-theoretic ordinal} of a theory~$T$ is the least ordinal~$\alpha$ such that $T$ does not prove its well-foundedness ($\WF(\alpha)$). For instance, the $\Pi^1_1$ proof-theoretic ordinals of~$\RCA_0$ and $\ACA_0$ are $\omega^\omega$ and~$\epsilon_0$, respectively (see Simpson~\cite{simpson_2009}).

The well-foundedness of ordinals naturally occured in the proof-theoretic analysis of Ramsey's theorem for pairs. Chong, Slaman and Yang~\cite{chong2017inductive} proved that $\RT^2_2$ does not imply the induction scheme for $\Sigma^0_2$ formulas ($\ISig_2$) over~$\RCA_0$ using a model of~$\RCA_0 + \BSig_2 + \neg \ISig_2$ which furthermore satisfied a principle called Bounded Monotone Enumeration ($\BME$).

\begin{definition}[\cite{chong2017inductive}]
Let $E$ be a procedure to recursively enumerate a finite branching enumerable tree. We say that $E$ is a \emph{monotone enumeration} if and only if the following conditions apply to its stage-by-stage behavior. 
\begin{enumerate}
    \item[(1)] The empty sequence is enumerated by E during stage 0.
    \item[(2)] Only finitely many sequences are enumerated by $E$ during any stage.
    \item[(3)] Suppose $T[s]$ is the tree enumerated by $E$ at the end of stage~$s$. Then any sequence enumerated by $E$ during stage $s + 1$ must extend some terminal node of $T[s]$.
\end{enumerate}
A monotone enumeration~$E$ is \emph{$b$-bounded} if every node in~$E$ has length at most~$b$.
\end{definition}

\begin{statement}
$\BME_{\star}$ is the statement \qt{Every bounded monotone enumeration is finite}.
\end{statement}

The statement $\BME$ is an iterated version of~$\BME_{\star}$. Its formal definition is rather technical and will not be used within this article.
Kreuzer and Yokoyama~\cite{kreuzer2016principles} proved that $\BME_{\star}$ is equivalent to $\WF(\omega^\omega)$ over~$\ISig_1$, and that $\BME$ is equivalent to $\bigcup_k \WF(\omega^\omega_k)$, where $\omega^\alpha_0 = \alpha$ and $\omega^{\alpha}_{k+1} = \omega^{\omega_k^\alpha}$. 

The principles $\BME$ and $\BME_\star$ naturally occur in the proofs of conservation of Ramsey's theorem for pairs. On the other hand, Patey and Yokoyama~\cite{patey2018proof} proved that $\RT^2_2$ is $\forall \Pi^0_3$-conservative over~$\RCA_0 + \BSig_2$, which implies in particular that the $\Pi^1_1$ proof-theoretic ordinal of~$\RT^2_2$ is $\omega^\omega$, hence that $\RT^2_2$ does not imply $\BME_\star$ over~$\RCA_0$.

In this article, we show that the separation of Chong, Slaman and Yang~\cite{chong2017inductive} of $\RT^2_2$ from~$\ISig_2$ only depends on the assumption that the original model satisfies $\BSig_2 + \BME$, by proving the following conservation theorems.

\begin{maintheorem}\label[maintheorem]{main:theorem1}
Fix~$n \geq 2$. Let~$T_n$ be either~$\BSig_n$ or~$\ISig_n$.
$\RCA_0 + T_n + \RT^2_2 + \WKL + \WF(\epsilon_0)$ is a $\Pi^1_1$-conservative extension of $\RCA_0 + T_n + \WF(\epsilon_0)$.
\end{maintheorem}

\begin{maintheorem}\label[maintheorem]{main:theorem2}
Fix~$n \geq 2$. Let~$T_n$ be either~$\BSig_n$ or~$\ISig_n$.
$\RCA_0 + T_n + \RT^2_2 + \WKL + \bigcup_{k \in \omega} \WF(\omega^\omega_k)$ is a $\Pi^1_1$-conservative extension of $\RCA_0 + T_n + \bigcup_{k \in \omega}\WF(\omega^\omega_k)$.
\end{maintheorem}

This strengthens a result from Ko{\l}odziejczyk and Yokoyama~\cite{kolo2021search} who showed that $\RCA_0 + \RT^2_2 + \BME$ is a $\Pi^0_4$-conservative extension of~$\RCA_0 + \BSig_2 + \BME$.
By Cholak, Jockusch and Slaman~\cite{cholak_jockusch_slaman_2001}, the result holds when~$\BSig_2 + \WF(\epsilon_0)$ is replaced by~$\ISig_2$. Thus, conservations over~$\RCA_0 + (\WF(\epsilon_0) \vee \ISig_2)$ and over~$\RCA_0 + (\bigcup_{n \in \omega}\WF(\omega^\omega_n) \vee \ISig_2)$ are the best known results.

In order to better understand the computational and proof-theoretic content of Ramsey's theorem for pairs, two main decompositions of $\RT^2_2$ into combinatorially weaker statements have been studied. These decompositions will be used to prove separate conservation theorems for simpler statements, and then combined to obtain our main theorems. We now detail those statements.

\subsection{Tournaments and linear orders}

Ramsey's theorem for pairs and two colors can be decomposed into the Erd\H{o}s-Moser theorem about infinite tournaments, and the Ascending Descending Sequence principle.

\begin{definition}
A \emph{tournament} on a domain~$D$ is an irreflexive binary relation $R \subseteq D^2$ such that for every~$a, b \in D$ with~$a \neq b$, exactly one of $R(a, b)$ and $R(b, a)$ holds. A tournament $R$ is \emph{transitive} if for every~$a, b, c \in D$, if $R(a, b)$ and $R(b, c)$ then $R(a, c)$.
\end{definition}

A subtournament of~$R \subseteq \NN^2$ is fully characterized by its domain and the parent tournament~$R$. We therefore identify subtournaments with infinite sets~$H \subseteq \NN$.
The following statement is known as the Erd\H{o}s-Moser theorem, and is an infinitary version of some theorem by Erd\H{o}s and Moser~\cite{erdos1964representation}.

\begin{statement}[Erd\H{o}s-Moser]
$\EM$ is the statement \qt{Every infinite tournament admits an infinite transitive subtournament}.
\end{statement}

There exists a one-to-one correspondence between a tournament~$R$ and a coloring $f : [\NN]^2 \to 2$, by letting $f(\{x, y\}) = 1$ if $(R(x, y) \leftrightarrow x <_\NN y)$. $\EM$ can the restated as \qt{for every coloring $f : [\NN]^2 \to 2$, there exists an infinite \emph{$f$-transitive} subset~$H$, that is, for every~$x, y, z \in H$ such that $x < y < z$ and every~$i < 2$, then if $f(\{x, y\}) = f(\{y, z\}) = i$, we have $f(\{x, z\}) = i$}.

\begin{statement}[Ascending Descending Sequence]
$\ADS$ is the statement \qt{Every infinite linear order admits an infinite ascending or descending subsequence}.
\end{statement}

Here again, there is a one-to-one correspondence between linear orders $\L = (\omega, <_\L)$, and transitive colorings $f : [\NN]^2 \to 2$, by letting $f(\{x, y\}) = 1$ if $(x <_\NN y \leftrightarrow x <_\L y)$. Then $\ADS$ can be restated as \qt{for every transitive coloring $f : [\NN]^2 \to 2$, there exists an infinite $f$-homogeneous set}.
Using the reformulations of $\EM$ and $\ADS$, the following proposition is immediate:

\begin{proposition}[Bovyking and Weiermann~\cite{bovykin2017strength}]
$\RCA_0 \vdash \RT^2_2 \leftrightarrow \EM \wedge \ADS$
\end{proposition}
\begin{proof}
$\EM$ and $\ADS$ are both particular cases of~$\RT^2_2$, hence $\RCA_0 \vdash \RT^2_2 \rightarrow \EM \wedge \ADS$. On the other hand, given a coloring $f : [\NN]^2 \to 2$, first apply $\EM$ to obtain an infinite $f$-transitive subset~$X = \{x_0 < x_1 < \dots \}$. Let~$g : [\NN]^2 \to 2$ be the transitive coloring defined by $g(a, b) = f(x_a, x_b)$. By~$\ADS$, there exists an infinite $g$-homogeneous subset~$H \subseteq \NN$. The set $\{ x_a : a \in H \}$ is infinite and $f$-homogeneous.
\end{proof}

From a combinatorial viewpoint, the decomposition of $\RT^2_2$ into $\EM$ and $\ADS$ divides the difficulty of Ramsey's theorem for pairs as follows: the combinatorics of $\EM$ are very close to those of $\RT^2_2$, but without the disjunctive nature of~$\RT^2_2$. Thus, there is no need to work with pairing arguments. On the other hand, $\ADS$ has very simple combinatorics, and does not imply any form of compactness, but is a disjunctive statement.

From a reverse mathematical viewpoint, the decomposition is non-trivial, in that $\EM$ and $\ADS$ are both strictly weaker than~$\RT^2_2$ (see Lerman, Solomon and Towsner~\cite{lerman2013separating}).

From a proof-theoretic viewpoint, both statements imply~$\BSig_2$ over~$\RCA_0$ (see Hirschfeldt and Shore~\cite[Proposition 4.5]{Hirschfeldt2007CombinatorialPW} and Kreuzer~\cite[Proposition 16]{kreuzer2012primitive}), and $\RCA_0 + \ADS$ is known to be $\Pi^1_1$-conservative extension of~$\RCA_0 + \BSig_2$ (see Chong, Slaman and Yang~\cite{chong2021pi11}). By an amalgamation theorem of Yokoyama~\cite{yokoyama2010conservativity}, it follows that $\RCA_0 + \RT^2_2$ is $\Pi^1_1$-conservative over~$\RCA_0 + \BSig_2$ iff so is~$\RCA_0 + \EM$.

\subsection{Cohesiveness and stability}

The second (and historically the first) decomposition of~$\RT^2_2$ was done by Cholak, Jockusch and Slaman~\cite{cholak_jockusch_slaman_2001}
in their seminal article on Ramsey's theorem for pairs. It consists in reducing $\RT^2_2$ to a stable version of it ($\SRT^2_2$), using the Cohesiveness principle ($\COH$). Then, it is easy to see that $\SRT^2_2$ is equivalent to the infinite pigeonhole principle for $\Delta^0_2$ instances, which is a combinatorialy simple statement.

\begin{definition}
A coloring $f : [\NN]^2 \to k$ is \emph{stable} if for every~$x \in \NN$, $\lim_y f(x, y)$ exists. Given an infinite sequence of sets~$\vec{R} = R_0, R_1, \dots$, an infinite set~$C$ is \emph{$\vec{R}$-cohesive} if for every~$s \in \NN$, $C \subseteq^{*} R_s$ or $C \subseteq^{*} \overline{R}_s$, where~$\subseteq^{*}$ denotes inclusion up to finite changes.
\end{definition}

\begin{statement}[Stable Ramsey's theorem for pairs]
$\SRT^2_k$ denotes the restriction of $\RT^2_k$ to stable colorings.
\end{statement}

\begin{statement}[Cohesiveness]
$\COH$ is the statement \qt{Every infinite sequence of sets has a cohesive set}.
\end{statement}

The following equivalence is essentially due to Cholak, Jockusch and Slaman~\cite{cholak_jockusch_slaman_2001}, with the proof $\RCA_0 \vdash \RT^2_2 \to \COH$ fixed by Mileti~\cite[Appendix A]{mileti2004partition}.

\begin{proposition}
$\RCA_0 \vdash \RT^2_2 \leftrightarrow \SRT^2_2 \wedge \COH$.
\end{proposition}
\begin{proof}
The proof of the forward direction will be omitted, as it will be of no use in the article. Let us prove that $\RCA_0 \vdash \SRT^2_2 \wedge \COH \to \RT^2_2$.
Given a coloring $f : [\NN]^2 \to 2$, define the sequence of sets~$\vec{R} = R_0, R_1, \dots$ by $R_x = \{ y : f(x, y) = 1 \}$. Given an infinite $\vec{R}$-cohesive set~$C = \{ x_0 < x_1 < \dots \}$, let $g : [\NN]^2 \to 2$ be defined by $g(\{a, b\}) = f(\{x_a, x_b\})$. By cohesiveness of~$C$, the coloring~$g$ is stable, so by $\SRT^2_2$, there is an infinite $g$-homogeneous set~$H$. 
The set $\{ x_a : a \in H \}$ is infinite and $f$-homogeneous.
\end{proof}

From a reverse mathematical viewpoint, the decomposition of~$\RT^2_2$ into~$\SRT^2_2$ and~$\COH$ is also non-trivial (see Hirschfeldt et al.~\cite{hirschfeldt2008strength} and Monin and Patey~\cite{monin2021srt22}).

The computability-theoretic content of the cohesiveness principle is well-understood: it admits a maximally complex instance, which is the sequence of all primitive recursive sets. A cohesive set for this sequence is called \emph{p-cohesive}. For every uniformly computable sequence of sets~$\vec{R}$, every $p$-cohesive set computes an $\vec{R}$-cohesive set. Moreover, a Turing degree bounds a p-cohesive set iff its jump is PA over~$\mathsf{0}'$. See Jockusch and Stephan~\cite{jockusch1993cohesive}.

The idea of using cohesiveness to transform a coloring into a stable one can be applied to $\EM$ and $\ADS$. We call $\SEM$ and $\SADS$ the restriction of $\EM$ and $\ADS$ to stable tournaments and linear orders, respectively.
Stable linear orders can be of three types: $\omega + \omega^{*}$, $\omega + k$ and $k + \omega^{*}$, for some~$k \in \omega$. The two former cases are computably true, so the only interesting case if whenever it is of order type~$\omega + \omega^{*}$. This is why~$\SADS$ is often defined as the statement \qt{Every linear order of type~$\omega + \omega^{*}$ admits an infinite ascending or descending sequence}.

\bigskip
Combining the two kind of decompositions, we shall use the following equivalence:

\begin{proposition}\label[proposition]{prop:rt22-sem-sads-coh}
$\RCA_0 \vdash \RT^2_2 \leftrightarrow \SEM \wedge \SADS \wedge \COH$
\end{proposition}

We will therefore prove conservation results about $\WKL, \COH, \SEM$ and $\SADS$ independently, and use the previous decomposition to obtain our main theorems.

\subsection{Definitions and notation}

Since $[X]^n$ is in one-to-one correspondence with the set of all $n$-tuples in increasing order, given a coloring $f : [\NN]^n \to k$, we might omit the brackets, and simply write $f(x_0, x_1, \dots, x_{n-1})$ instead of $f(\{x_0, x_1, \dots, x_{n-1}\})$, assuming that $x_0 < \dots < x_{n-1}$.

A \emph{$\Pi^1_2$ problem} $\mathsf{P}$ is a formula $\forall X[\varphi(X) \rightarrow \exists Y \psi(X, Y)]$ where $\varphi$ and $\psi$ are arithmetic formulas. A set~$X$ such that $\varphi(X)$ holds is a \emph{$\mathsf{P}$-instance}, and a $Y$ such that $\psi(X, Y)$ holds is a \emph{$\mathsf{P}$-solution} to~$X$. In particular, $\RT^2_2$, $\EM$ and $\ADS$ are $\Pi^1_2$-problems. For example, an $\RT^2_2$-instance is a coloring~$f : [\NN]^2 \to 2$
and an $\RT^2_2$-solution to~$f$ is an infinite $f$-homogeneous set.
\bigskip

\textbf{Binary strings}. We let $2^{<\NN}$ denote the set of all finite binary strings, and $2^\NN$ denote the class of all infinite binary sequences. Note that $2^\NN$ is in bijection with $\P(\NN)$ and both are usually identified.
Finite binary strings are written with greek letters~$\sigma, \tau, \mu, \dots$. The length of a binary string $\sigma$ is written~$|\sigma|$, and the concatenation of two binary strings $\sigma$ and $\tau$ is written $\sigma \cdot \tau$. Given an infinite binary sequence~$X \in 2^\NN$ and $n \in \NN$, we write $X \uh_n$ for its initial segment of length~$n$. A string~$\sigma$ is a prefix of a string~$\tau$, written $\sigma \preceq \tau$, if there is some~$\mu$ such that $\sigma \cdot \mu = \tau$.
\bigskip

\textbf{Finite sets}. Finite binary strings, seen as finite characteristic functions, are often identified with finite sets, that is, $\sigma$ is identified with $\{ n < |\sigma| : \sigma(n) = 1 \}$. Based on this correspondence, we extend the set-theoretic notations to binary strings, and let for example $\sigma \cup \rho$ denote the binary string of length~$\max(|\sigma|, |\rho|)$, and such that $(\sigma \cup \rho)(n) = 1$ iff $\sigma(n) = 1$ or $\rho(n) = 1$.
In particular, one shall distinguish the cardinality $\card \sigma = \card \{ n < |\sigma| : \sigma(n) = 1 \}$ of a string seen as a set, from its length~$|\sigma|$.
\bigskip

\textbf{Theories.}
Let $\ISig_n$ be the following induction scheme for every~$\Sigma^0_n$ formula~$\phi$:
$$
\phi(0) \wedge \forall x[\phi(x) \rightarrow \phi(x+1)] \rightarrow \forall x \phi(x)
$$
Let~$\Delta^0_n$-$\mathsf{CA}$ be the following comprehension scheme for every~$\Sigma^0_n$ formula~$\phi$ and every~$\Pi^0_n$ formula~$\psi$:
$$
\forall x[\phi(x) \leftrightarrow \psi(x)] \rightarrow \exists A \forall x[x \in A \leftrightarrow \phi(x)]
$$
$\RCA_0^*$ denotes the theory of Robinson arithmetic ($Q$) augmented with the comprehension axiom for~$\Delta^0_1$ predicates ($\Delta^0_1$-$\mathsf{CA}$), the collection scheme for $\Sigma^0_1$ formulas ($\BSig_1$) and the totality of the exponentiation function~($\mathsf{exp}$).
$\RCA_0$ is the theory of $Q + \Delta^0_1\mbox{-}\mathsf{CA} + \ISig_1$. Note that $\RCA_0$ entails~$\RCA_0^*$. The theory~$\RCA_0$ is the base theory of Reverse Mathematics, but $\RCA_0^*$ is sometimes more convenient to use, as shown by Belanger~\cite{Blanger2022ConservationTF}.
\bigskip

\textbf{Models.}
A \emph{model} is a second-order structure $\M = (M, S, +_M, \times_M, <_M)$, where $M$ is the first-order part, that is, the set of integers and $S \subseteq \P(M)$ is the second-order part. We will usually omit the operations, and simply write $\M = (M, S)$.
A model $\M = (M, S)$ is \emph{topped} if there is some set~$X \in S$ such that every~$Y \in S$ is $\Delta^0_1(X)$. A model $\Nc = (N, T)$ is an \emph{$\omega$-submodel} of $\M = (M, S)$ if $N = M$ and $T \subseteq S$. Note that the notion of $\omega$-submodel must not be confused with the notion of $\omega$-model, in which $M = \omega$, that is, is the set of all standard integers.

Given a model~$\M = (M, S)$ and a set~$G \subseteq M$, we write $\M[G]$ for the model whose first-order part is $M$, and whose second-order part consists of all $\Delta^0_1(X, G)$ sets for~$X \in S$. Therefore, if $\M \models \RCA_0$, and $\M \cup \{G\} \models \ISig_1$, then $\M[G] \models \RCA_0$.

A set~$X$ is \emph{$M$-bounded} if there is some~$x \in M$ such that $\forall y \in X,\ y \leq x$.
A set~$X$ is \emph{$M$-finite} or \emph{$M$-coded} if there exists some~$x \in M$ such that $x = \sum_{n \in X} 2^n$.

Given two sets~$X$ and $Y$, we write $X \gg Y$ to say that $X$ is of PA degree over~$Y$. More formally, it means that $\Delta^0_1(Y)$ infinite binary tree admits a $\Delta^0_1(X)$ path.
\bigskip

\textbf{Ordinals.}
In this article, we shall work with ordinals smaller than~$\epsilon_0$ in models of weak arithmetic.
By Cantor's normal form theorem, every ordinal $\alpha$ below $\epsilon_0$ can be written uniquely as $\alpha = \omega^{\gamma_1}n_1 + \dots + \omega^{\gamma_k}n_k$ for some ordinals $\gamma_k < \dots < \gamma_1 < \alpha$ and positive integers $n_1,\dots, n_k$. So, every such ordinal can be represented as a finite rooted tree and therefore be coded by an integer which will be called its "index". For a suitable encoding of the finite rooted trees, the cantor normal form of any ordinal below $\epsilon_0$ and the standard operations on ordinals (sums, product, comparaison) are computable from their indexes. See H\'ajek and Pudl\'ak~\cite[Chapter~II.3]{hajek2017metamathematics} for a formal study of ordinals in the context of weak arithmetic.
 
\subsection{Structure of the paper}

In \Cref{sect:bsig-wfalpha}, we develop some basic techniques about $\Pi^1_1$-conservation over~$\RCA_0 + \BSig_2 + \WF(\alpha)$. We prove in particular that $\RCA_0 + \BSig_2 + \WF(\alpha) + \Gamma$ is a $\Pi^1_1$-conservative extension of~$\RCA_0 + \BSig_2 + \WF(\alpha)$ for~$\Gamma \in \{\COH, \WKL\}$ and $\alpha$ any primitive recursive ordinal.

In \Cref{sect:bsig2-epsilon0}, we handle the case of~$\SADS$ and~$\SEM$ which does not admit such a direct preservation of well-foundedness. In these cases, one needs to assume $\WF(\sup_{\beta < \alpha} \omega^{\beta \times \omega})$ to obtain an $\omega$-extension satisfying $\WF(\alpha)$, where~$\alpha \leq \epsilon_0$. We prove in particular that $\RCA_0 + \BSig_2 + \WF(\epsilon_0) + \Gamma$ is a $\Pi^1_1$-conservative extension of~$\RCA_0 + \BSig_2 + \WF(\epsilon_0)$ for~$\Gamma \in \{\SADS, \SEM\}$.

In \Cref{sect:conservation-rt22}, we combine the previous conservation theorems to obtain our main theorems.
Last, in \Cref{sect:open-questions}, we state some remaining open questions.

\section{Conservation over $\BSig_2 + \WF(\alpha)$}\label[section]{sect:bsig-wfalpha}

In this section, we introduce the framework for preserving $\WF(\alpha)$ for an ordinal~$\alpha \leq \epsilon_0$ over~$\RCA_0 + \BSig_2$. We then prove that $\RCA_0 + \BSig_2 + \WF(\alpha) + \WKL$ and $\RCA_0 + \BSig_2 + \WF(\alpha) + \COH$ are both~$\Pi^1_1$-conservative extensions of $\RCA_0 + \BSig_2 + \WF(\alpha)$.

The proofs of conservation will all follow the same pattern: Given a countable topped model  $\M = (M,S) \models \RCA_0 + \BSig_2 + \WF(\alpha)$, and an instance $I \in S$ of our problem $\mathsf{P}$, we will show the existence of a $\mathsf{P}$-solution $G \subseteq M$ of~$I$ such that $\M[G] \models \RCA_0 + \BSig_2 + \WF(\alpha)$. Then, the conservation result follows from a standard argument that we give for the sake of completeness:



\begin{theorem}\label[theorem]{theorem:problem-wf-cons}
Let $\mathsf{P}$ be a $\Pi_2^1$ problem and $\alpha \leq \epsilon_0$ be an ordinal such that for every countable topped model $\M = (M,S) \models \RCA_0 + \BSig_2 + \WF(\alpha)$ and for every $\mathsf{P}$-instance $I \in \M$, there exists a $\mathsf{P}$-solution $G \subset M$ of $I$ such that $\M[G] \models \RCA_0 + \BSig_2 + \WF(\alpha)$. Then $\RCA_0 + \BSig_2 + \WF(\alpha) + \mathsf{P}$ is $\Pi_1^1$-conservative over $\RCA_0 + \BSig_2 + \WF(\alpha)$.
\end{theorem}

\begin{proof}
Assume $\RCA_0 + \BSig_2 + \WF(\alpha) \not \vdash \forall X \phi(X)$ for $\phi$ an arithmetic formula. Then by the completeness theorem and the downward Löwenheim–Skolem theorem, there is a countable model $\M = (M,S) \models \RCA_0 + \BSig_2 + \WF(\alpha) + \neg \phi(A)$ with $A \in S$. We can furthermore assume that $\M$ is topped by $A$.

From our assumption, we can build a sequence $\M = \M_0 \subseteq \M_1 = (M,S_1) \subseteq \M_2 = (M,S_2) \subseteq \dots$ of countable topped model of $\RCA_0 + \BSig_2 + \WF(\alpha)$ such that every $\mathsf{P}$-instance appearing in one of the $\M_i$ will eventually have a solution in one $\M_j$. Then $\bigcup_{n < \omega} \M_n$ is a model of $\RCA_0 + \BSig_2 + \WF(\alpha) + \mathsf{P} + \neg \phi(A)$. So $\RCA_0 + \BSig_2 + \WF(\alpha) + \mathsf{P} \not \vdash \forall X \phi(X)$.
\end{proof}

\subsection{Conservation over~$\BSig_2$}

There exist two main methods for extending a model $\M$ into a model $\M[G]$: by external forcing, and by an internal construction. In the case of conservation theorems over~$\BSig_2 + \neg \ISig_2$, one arguably needs to use an internal construction. Indeed, using an isomorphism theorem for~$\WKL^*_0$, Fiori-Carones et al~\cite{fiori2021isomorphism} proved that $\Pi^1_1$-conservation of any $\Pi^1_2$ problem over~$\BSig_2 + \neg \ISig_2$ is equivalent to a formalized version of the first-jump control.

Let~$\IDel_n$ be the following induction scheme for every~$\Sigma^0_n$ formula~$\phi$ and every~$\Pi^0_n$ formula~$\psi$:
$$
\forall x[\phi(x) \leftrightarrow \psi(x)] \wedge \phi(0) \wedge \forall x[\phi(x) \rightarrow \phi(x+1)] \rightarrow \forall x \phi(x)
$$
A set~$X$ is \emph{$M$-regular} (or \emph{piecewise-coded}, or even \emph{amenable}) if all its initial segments are $M$-finite.
Over this article, we will use the following characterizations of~$\BSig_2$:

\begin{proposition}[Chong and Mourad~\cite{chong1990degree} and Slaman~\cite{slaman2004bounding}]
The following are equivalent over~$\RCA_0$:
\begin{enumerate}
    \item[(1)] $\BSig_2$
    \item[(2)] The induction principle for $\Delta^0_2$ predicates ($\IDel_2$)
    \item[(3)] Every $\Delta^0_2$ set is regular
\end{enumerate}
\end{proposition}

Given a countable model~$\M = (M, S) \models \RCA_0 + \BSig_2 + \WF(\alpha)$ for some~$\alpha \leq \epsilon_0$ topped by a set~$Y \in S$, and a $\mathsf{P}$-instance~$I \in S$ for some $\Pi^1_2$ problem~$\mathsf{P}$, we will create a decreasing sequence of conditions $c_0 \geq c_1 \geq \dots$ (for the appropriate notion of forcing depending on~$\mathsf{P}$) such that
\begin{enumerate}
    \item[(a)] The sequence is $\Delta^0_1(P)$ for some set~$P \subseteq M$ such that $\M[P] \models \IDel_1$
    \item[(b)] The set of~$s \in M$ such that $c_s$ is defined is a cut
    \item[(c)] For every~$k \in M$, there is some~$s \in M$ such that $c_s$ decides $(G \oplus Y)' \uh_k$
    \item[(d)] For every~$k \in M$, there is some~$s \in M$ such that $c_s$ forces $\Phi_e^{G \oplus Y}$ not to be an infinite $\alpha$-decreasing sequence for every~$e < k$
\end{enumerate}

The item (a) will ensure that there is a $\Delta^0_1(P)$ formula $\psi(s)$ stating that the construction can be pursued up to stage~$s$. By item (b), $\{ s : \psi(s) \}$ is a cut, so by $\IDel_1(P)$, the construction can be pursued at every stage. By item (c), $(G \oplus Y)'$ will be $\Delta^0_1(P)$, hence $\M[G] \models \BSig_2$. Last, by item (d), $\M[G] \models \WF(\alpha)$.

The use of an internal construction prevents from using a bijection between $M$ and $\omega$ to satisfy the requirements one by one. Therefore, one must use \emph{Shore blocking} arguments, that is, force the requirements simultaneously for larger and larger initial segments of~$M$, as mentioned in items~(c) and~(d).

\subsection{Well-foundedness and Shore blocking}

Before working on $\Pi^1_1$-conservation of~$\WKL$ and $\COH$, we prove two combinatorial lemma which say that Shore blocking comes for free for preserving $\WF(\alpha)$ for any ordinal $\alpha \leq \epsilon_0$. Instead of using the standard sums and products over ordinals, the natural (Hessenberg) sums and products will be more convenient:

\begin{definition}[Natural sum and product]
Let $\alpha$ and $\beta$ be two ordinals less than $\epsilon_0$. Let $\alpha = \omega^{\gamma_1}n_1 + \dots + \omega^{\gamma_k}n_k$ and $\beta = \omega^{\gamma_1}m_1 + \dots + \omega^{\gamma_k}m_k$ (We allow the $n_i$ and $m_i$ to be equal to $0$ in order to write $\alpha$ and $\beta$ using the same exponents $\gamma_i$). We define the natural sum of $\alpha$ and $\beta$ (denoted by $\alpha \nplus \beta$) to be the following ordinal (It does not depend on the choice of the family $(\gamma_i)$) : 

$$\omega^{\gamma_1}(n_1 + m_1) + \dots + \omega^{\gamma_k}(n_k + m_k)$$

We also define $\alpha \ntimes k$ to be equal to be the natural sum of $\alpha$ with itself $k$ times. 
\end{definition}

Note that $\alpha \ntimes k \leq \alpha \times \omega$, where $\alpha \times \omega = \sup_{k < \omega} \alpha \times k = \omega^{\gamma_1 + 1}n_1$ is the usual ordinal product.

\begin{proposition}
For all $\beta < \epsilon_0$, and $\alpha' < \alpha < \epsilon_0$ : $\alpha' \nplus \beta < \alpha \nplus \beta$ and $\beta \nplus \alpha' < \beta \nplus \alpha$.
\end{proposition}

\begin{proof}
Clear from the definition of $\nplus$.
\end{proof}

\begin{lemma}\label[lemma]{lem:wf-product}
Let~$\M = (M,S) \models \RCA_0 + \WF(\alpha)$ for $\alpha$ an ordinal less than $\epsilon_0$. Then $\M \models \WF(\alpha \ntimes k)$ for every $k \in M$.
\end{lemma}

\begin{proof}
Since $\alpha \ntimes k \leq \alpha \times \omega$ for every $k \in M$ and $\RCA_0$ proves that the product of two well-orders is a well-order, hence $\RCA_0 \vdash \WF(\alpha) \rightarrow \WF(\alpha \times \omega)$, we get the desired result.



\end{proof}

\begin{lemma}\label[lemma]{lem:shore-blocking-wf}
Consider a model $\M = (M,S)$ and $\alpha \in M$ an ordinal less than $\epsilon_0$.
For every~$k \in M$, there is a Turing functional $\Gamma_k$ such that, letting~$\beta \in \alpha$ be the largest ordinal less than $\alpha$ of index at most~$k$, for every~$X \in 2^M$ such that $\M[X] \models \RCA_0$, if there is some~$e < k$ such that $\Phi^X_e$ is an $M$-infinite decreasing sequence of elements smaller than~$\beta$, then $\Gamma^X_k$ is an $M$-infinite decreasing sequence of elements smaller than~$\beta \ntimes k$.

Moreover, an index of~$\Gamma_k$ can be found computably in~$k$.
\end{lemma}

\begin{proof}
By twisting the Turing functionals, we can assume that for every~$e, n \in M$, if $\Phi^\sigma_e(n)\downarrow$, then
\begin{itemize}
    \item[(1)] $n < |\sigma|$ ;
    \item[(2)] $\Phi^\sigma_e(m)\downarrow$ for every~$m < n$ ;
    \item[(3)] $\Phi^\sigma_e(0), \Phi^\sigma_e(1), \dots, \Phi^\sigma_e(n)$ is a strictly decreasing sequence of elements smaller than~$\beta$.
\end{itemize}

Given $\sigma \in 2^{<M}$ and $e < k$, let $\alpha(\sigma, e) = \Phi^{\sigma}_e(s)$ for the largest $s < |\sigma|$ such that $\Phi^{\sigma}_e(s)\downarrow$. If there is no such~$s$, then $\alpha(\sigma, e) = \beta$. Note that if $\sigma' \succeq \sigma$, then $\alpha(\sigma',e) \leq \alpha(\sigma,e)$.

Let~$\sigma_{-1} = \epsilon$.
Let~$\Gamma^X_k$ be the Turing functional which, on input~$n$, searches for some~$x > |\sigma_{n-1}|$ and some~$\sigma_n \prec X$ such that $\Phi_e^{\sigma_n}(x)\downarrow$ for some~$e < k$. If found, it outputs $\alpha(\sigma, 0) \nplus \dots \nplus \alpha(\sigma, k-1)$. Note that if~$\Gamma^X_k(n)\downarrow$, then by (3), $\Gamma^X(n)$ is an ordinal smaller than $\beta \ntimes k$.

Suppose that $X$ is such that $\M[X] \models \RCA_0$ and there is an $e < k$ is such that $\Phi_e^X$ is total.
Since~$\M[X] \models \ISig_1$, $\Gamma^X$ is total.

Moreover, since~$x > |\sigma_{n-1}|$, letting~$e < k$ be such that $\Phi_e^{\sigma_n}(x)\downarrow$, by (1) we have $\Phi_e^{\sigma_{n-1}}(x)\uparrow$, so by (2) and (3), $\alpha(\sigma_{n+1},e) < \alpha(\sigma_n, e)$, hence $\Gamma^X_k(n+1) < \Gamma^X_k(n)$. It follows that $\Gamma^X_k$ is an $M$-infinite decreasing sequence of ordinals smaller than~$\beta \ntimes k$.
\end{proof}

\subsection{$\Pi^1_1$-conservation of~$\COH$}

We now turn to the proof that $\RCA_0 + \BSig_2 + \WF(\alpha) + \COH$ is $\Pi^1_1$-conservative of $\RCA_0 + \BSig_2 + \WF(\alpha)$ for every ordinal~$\alpha \leq \epsilon_0$. In particular, letting~$\alpha = \omega$, we reprove the theorem of Chong, Slaman and Yang~\cite{chong2021pi11} stating that $\RCA_0 + \BSig_2 + \COH$ is $\Pi^1_1$-conservative over~$\RCA_0 + \BSig_2$.
Note that the original proof by Chong, Slaman and Yang uses a case analysis, depending on whether the model satisfies $\ISig_2$ or not. Belanger~\cite{Blanger2022ConservationTF} gave a more direct proof using characterization of p-cohesive sets are those whose jump is PA over~0' (see Jockusch and Stephan~\cite{jockusch1993cohesive}). Our proof follows the ideas from Belanger, but with an even more direct construction based on an instance-wise correspondence between instances of~$\COH$ and $\Pi^0_1(\emptyset')$ classes. As in Belanger's proof, the proof relies on the following theorem:

\begin{theorem}[Simpson–Smith~\cite{simpson1986factorization}]\label[theorem]{thm:wkl-rca0s}
Every countable model of~$\RCA_0^*$ can be $\omega$-extended into a model of~$\RCA_0^* + \WKL$.
\end{theorem}

\begin{proposition}\label[proposition]{prop:coh-wf}
Consider a countable model $\M = (M,S) \models \RCA_0 + \BSig_2 + \WF(\alpha)$ topped by a set $Y \in S$, where $\alpha \leq \epsilon_0$ is an ordinal. Then for every countable sequence of sets $\vec{R} \in S$ and every~$P \gg Y'$ such that $\M[P] \models \RCA_0^*$, there exists $G \subseteq M$ such that 
\begin{enumerate}
    \item $G$ is $\vec{R}$-cohesive ;
    \item $P \geq_T (G \oplus Y)'$ ;
    \item $\M[G] \models \RCA_0 + \BSig_2 + \WF(\alpha)$.
\end{enumerate}
\end{proposition}
\begin{proof}
Given $\vec{R} = R_0, R_1, \dots$ and $\sigma \in 2^{<M}$, let $\vec{R}_\sigma = \bigcap_{\sigma(i) = 0} \overline{R}_i \bigcap_{\sigma(i) = 1} R_i$. Let $T$ be the $\Sigma^0_1(Y)$ tree $\{ \sigma \in 2^{<M} : \exists x (x > |\sigma| \wedge x \in \vec{R}_\sigma) \}$. Note that a node $\sigma$ is extendible in $T$ iff $\vec{R}_\sigma$ is infinite. The tree 
$T$ is $\Delta^0_1(Y')$, so since~$P \gg Y'$, there is a $\Delta^0_1(P)$ set~$Q \in [T]$. 

We will define the set~$G$ using an infinite $\Delta^0_1(P)$ decreasing sequence of Mathias-like conditions. 

\begin{definition}
A \emph{condition} is a pair of $M$-finite binary strings $(\sigma,\rho)$ such that $\rho \prec Q$.
We have $(\sigma_2,\rho_2) \leq (\sigma_1,\rho_1)$ if $\sigma_2 \succ \sigma_1$, $\rho_2 \succ \rho_1$ and $\sigma_2 - \sigma_1 \subseteq \vec{R}_{\rho_1}$ (where $\sigma_2 - \sigma_1$ is the difference between the sets corresponding to the two sequences).
\end{definition}

One can think of a condition $(\sigma, \rho)$ as a Mathias condition $(\sigma, \vec{R}_\rho)$. Requiring that $\rho \prec Q$ ensures that $\vec{R}_\rho$ is $M$-infinite. When increasing the length of~$\rho$, we make progress in cohesiveness.
Note that being a condition is a $\Delta^0_1(Q)$ predicate, as it simply requires checking that $\sigma \prec Q$, while the extension relation is $\Delta^0_1(Y)$. In particular, both are $\Delta^0_1(Q \oplus Y)$, so $\Delta^0_1(P)$.

A sequence $\tau$ is said to be \emph{compatible} with a condition $(\sigma,\rho)$ if $\sigma \prec \tau$ and  $\tau - \sigma \subseteq \vec{R}_{\rho}$.

\begin{definition}
Let~$(\sigma, \rho)$ be a condition, and $\Phi_e$ be a Turing functional.
\begin{enumerate}
    \item[1.] $(\sigma, \rho) \Vdash \exists x \Phi^{G \oplus Y}_e(x)\downarrow$ if there is some~$x, t < |\sigma|$ such that $\Phi^{\sigma \oplus Y}_e(x)[t]\downarrow$
    \item[2.] $(\sigma, \rho) \Vdash \forall x \Phi^{G \oplus Y}_e(x)\uparrow$ if for every~$\tau$ compatible with~$(\sigma, \rho)$ and every~$x \in M$, $\Phi^{\tau \oplus Y}_e(x)\uparrow$.
\end{enumerate}
\end{definition}

Note that the relations $(\sigma, \rho) \Vdash \exists x \Phi^{G \oplus Y}_e(x)\downarrow$
and $(\sigma, \rho) \Vdash \forall x \Phi^{G \oplus Y}_e(x)\uparrow$ are $\Delta^0_0(Y)$ and $\Pi^0_1(Y)$, respectively. Also note that since the formulas $\exists x \Phi^{G \oplus Y}_e(x)\downarrow$ and $\forall x \Phi^{G \oplus Y}_e(x)\uparrow$ are universally $\Sigma^0_1(G \oplus Y)$ and $\Pi^0_1(G \oplus Y)$, respectively, the definition above induces a forcing relation for these classes of formulas.

To ensure that $P \geq_T (G \oplus Y)'$, we will also construct $(G \oplus Y)'$ in parallel. \\

For this, we will need to satisfy three kind of requirements for every~$k \in M$:
\begin{itemize}
    \item $\mathcal{R}_k$: $(G \oplus Y)' \uh_k$ is decided
    \item $\mathcal{S}_k$: for every~$e < k$, $\Phi^G_e$ is not an $M$-infinite $\alpha$-decreasing sequence
    \item $\mathcal{T}_k$: for every~$e < k$, $G \subseteq^* R_e$ or $G \subseteq^* \overline{R}_e$
\end{itemize}

We will need the following two technical lemmas, which are Shore blocking arguments for satisfying the two kind of requirements, namely, deciding the jump and preserving $\WF(\alpha)$. In what follows, $(\tau, \rho) \Vdash (G \oplus Y)' \uh_k = \sigma'$ means that for every~$e < |\sigma'|$, if $\sigma'(e) = 1$ then $(\tau, \rho) \Vdash \Phi^{G \oplus Y}_e(e)\downarrow$ and if $\sigma'(e) = 0$, then $(\tau, \rho) \Vdash \Phi^{G \oplus Y}_e(e)\uparrow$. Thus, this relation is $\Pi^0_1(Y)$.

\begin{lemma}\label[lemma]{lem:coh-wf-forcing-jump}
Let $(\sigma, \rho)$ be a condition. For every~$k \in M$, there exists some $M$-finite $\sigma' \in 2^k$ and an extension $(\tau, \rho) \leq (\sigma, \rho)$ such that $(\tau, \rho) \Vdash (G \oplus Y)' \uh_k = \sigma'$.
\end{lemma}
\begin{proof}
Let~$W$ be the set of all $\sigma' \in 2^k$ such that $(\exists \tau \succ \sigma)(\exists t)(\forall e < k)(\tau - \sigma \in \vec{R}_\rho \wedge |\tau - \sigma| \geq 1 \wedge \sigma'(e) = 1 \to \Phi_e^{\sigma \oplus Y}(e)[t] \downarrow)$.
Note that $W$ is $\Sigma^0_1(Y)$, hence is $M$-finite. Moreover, $W$ is non-empty, as it contains the string $000\dots$. Let~$\sigma' \in W$ be the lexicographically maximal element, and let~$\tau \succeq \sigma$ witness that $\sigma' \in W$.

We claim that $(\tau, \rho)$ forces $(G \oplus Y)' \uh_k = \sigma'$.
Fix some $e < k$.

Case (a): $\sigma'(e) = 1$. Then $\Phi_e^{\tau \oplus Y}(e)\downarrow$, hence for every~$G$ compatible with $(\tau, \rho)$, $e \in (G \oplus Y)'$.

Case (b): $\sigma'(e) = 0$. The maximality of $\sigma'$ ensures that for every~$\mu$ compatible with $(\tau, \rho)$, $\Phi^{\mu \oplus Y}(e)\uparrow$. It follows that $\Phi^{G \oplus Y}_e(e)\uparrow$, hence $e \not \in (G \oplus Y)'$.
\end{proof}

\begin{lemma}\label[lemma]{lem:coh-wf-forcing-wf}
Let $(\sigma, \rho)$ be a condition. For every~$k \in M$, there exists some~$n \in M$ an extension $(\tau, \rho) \leq (\sigma, \rho)$ forcing~$\Phi_e^{G \oplus Y}$ not to be an $\alpha$-decreasing sequence of length more than~$n$ for every~$e < k$.
\end{lemma}
\begin{proof}
By twisting the Turing functionals, we can assume that for every~$e, n \in M$, if $\Phi^\sigma_e(n)\downarrow$, then
\begin{itemize}
    \item $\Phi^\sigma_e(m)\downarrow$ for every~$m < n$
    \item $\Phi^\sigma_e(0), \Phi^\sigma_e(1), \dots, \Phi^\sigma_e(m)$ is a strictly $\alpha$-decreasing sequence
\end{itemize}

Suppose the lemma does not hold. Then in particular, for every~$M$-finite string~$\tau$ compatible with $(\sigma, \rho)$ and every~$n \in M$, there is some~$\tau_1 \succeq \tau$ compatible with~$(\sigma, \rho)$ and some~$e < k$ such that $\Phi^{\tau_1 \oplus Y}(n)\downarrow$.

Build $Y$-recursively an infinite sequence $\tau_0 \preceq \tau_1 \preceq \tau_2 \preceq \dots$ such that for every~$n$,
\begin{itemize}
    \item $\tau_n$ is compatible with $(\sigma, \rho)$
    \item $\Phi^{\tau_n \oplus Y}_e(n)\downarrow$ for some~$e < k$ (which depends on~$n$)
\end{itemize}
Since~$\M \models \BSig_2$, $\M \models \forall k\RT^1_k$,
so there is an $M$-infinite set~$Z \in \M$ and some~$e < k$ such that for every~$n \in Z$, $\Phi^{\tau_n \oplus Y}_e(n)\downarrow$.

Then, letting $f(n) = \Phi^{\tau_n \oplus Y}_e(n)$, we have an $M$-infinite $\alpha$-decreasing sequence, contradicting $\M \models \WF(\alpha)$.
\end{proof}

Note that the statement "$(\tau, \rho)$ forces $\Phi_e^{G \oplus Y}$ not to be an $\alpha$-decreasing sequence of length more than~$n$ for every~$e < k$"
is $\Pi^0_1(Y)$. Indeed, it means that for every~$\mu$ compatible with $(\tau, \rho)$ and every~$e < k$, $\Phi_e^{\mu \oplus Y}(n)\uparrow$."
\bigskip

\textbf{Construction}.
We will build a decreasing sequence $(\sigma_s, \rho_s)$ of conditions and then take for $G$ the union of the $\sigma_s$. We will also build an increasing sequence $(\sigma'_s)$ such that $(G \oplus Y)'$ will be the union of the $\sigma'_s$. Initially, we take $\sigma_0 = \sigma'_0 = \rho_0 = \epsilon$, and during the construction, we will ensure that we have $|\sigma_s|, |\sigma'_s|, |\rho_s| \leq s$ at every stage. Each stage will be either of type $\R$, of type $\S$ or of type~$\T$. The stage $0$ is of type $\R$.

Assume that $(\sigma_{s}, \rho_{s})$ and $\sigma'_{s}$ are already defined. Let~$s_0 < s$ be the latest stage at which we switched the stage type. We have three cases.

Case 1: $s$ is of type~$\R$. If there exists some~$\tau \in 2^{\leq s}$ and some~$\sigma' \in 2^{s_0}$ such that $(\tau, \rho_s) \leq (\sigma_s, \rho_s)$, and 
$(\tau, \rho_s)$ forces $(G \oplus Y)' \uh_{s_0} = \sigma'$, then let~$\sigma_{s+1} = \tau$, $\rho_{s+1} = \rho_s$, $\sigma'_{s+1} = \sigma'$ and let~$s+1$ be of type~$\S$. Otherwise, the elements are left unchanged and we go to the next stage.

Case 2: $s$ is of type~$\S$. If there exists some~$\tau \in 2^{\leq s}$ such that $(\tau, \rho_s) \leq (\sigma_s, \rho_s)$, and $(\tau, \rho_s)$ forces $\S_{s_0}$, then let~$\sigma_{s+1} = \tau$, $\rho_{s+1} = \rho_s$, $\sigma'_{s+1} = \sigma'_s$ and let~$s+1$ be of type~$\T$. Otherwise, the elements are left unchanged and we go to the next stage.

Case 3: $s$ is of type~$\T$. If there exists some $\tau \in 2^{\leq s}$ such that $(\tau, \rho_s) \leq (\sigma_s, \rho_s)$ and $\tau - \sigma_s \neq \emptyset$. Then let~$\sigma_{s+1} = \tau$, $\rho_{s+1} = P \uh_{s_0}$, $\sigma'_{s+1} = \sigma'_s$, and let~$s+1$ be of type~$\R$. Otherwise, the elements are left unchanged and we go to the next stage.

This completes the construction.
\bigskip

\textbf{Verification}.
Since the size of $\sigma_s, \sigma'_s$ and $\rho_s$ are bounded by $s$, there is a $\Delta_1^0(Y' \oplus Q)$-formula $\phi(s)$ stating that the construction can be pursued up to stage $s$. Our construction implies that the set $\{s|\phi(s)\}$ is a cut, so by $\IDel_1(P)$, the construction can be pursued at every stage. 

Let~$G = \bigcup_{s \in M} \sigma_s$.
By \Cref{lem:coh-wf-forcing-jump} and \Cref{lem:coh-wf-forcing-wf} and the fact that each $\vec{R}_{\rho_s}$ is $M$-infinite,
each type of stage changes $M$-infinitely often. Thus, $G$ is $M$-infinite and $\{|\sigma'_s| : s \in M \}$ is $M$-infinite, so $P \geq_T (G \oplus Y)'$. In particular, $\M[G] \models \RCA_0 + \BSig_2$. Moreover, for every~$e \in M$, there is a stage~$s_0 > e$ at which the stage type switches from $\R$ to $\S$, and a stage~$s > s_0$ at which the stage switches from~$\S$ to~$\T$.
Then, $(\sigma_{s+1}, \rho_{s+1})$ forces $\S_{s_0}$, so $\Phi^{G \oplus Y}_e$ is not an $M$-infinite $\alpha$-decreasing sequence. Thus, $\M[G] \models \WF(\alpha)$.
Similarly, for every~$e \in M$, there is a stage~$s > e$ such that the stage type switches from~$\S$ to~$\T$. Then $(\sigma_{s+1}, \rho_{s+1})$ is such that $|\rho_{s+1}| > e$, so $G \subseteq^{*} R_e$ or $G \subseteq^{*} \overline{R}_e$. Thus $G$ is $\vec{R}$-cohesive. This completes the proof of \Cref{prop:coh-wf}.
\end{proof}

\begin{corollary}\label[corollary]{cor:coh-wf}
Consider a countable model $\M = (M,S) \models \RCA_0 + \BSig_2 + \WF(\alpha)$ topped by a set $Y \in S$, where $\alpha \leq \epsilon_0$ is an ordinal. Then for every $\vec{R} \in S$, there exists $G \subseteq M$ such that 
\begin{enumerate}
    \item $G$ is $\vec{R}$-cohesive ;
    \item $\M[G] \models \RCA_0 + \BSig_2 + \WF(\alpha)$.
\end{enumerate}
\end{corollary}
\begin{proof}
By \Cref{thm:wkl-rca0s}, there exists a set~$P \subseteq M$ such that $P \gg Y'$ and $\M[P] \models \RCA_0^*$.
The conclusion follows from~\Cref{prop:coh-wf}.
\end{proof}





We therefore obtain our first conservation theorem:

\begin{theorem}\label[theorem]{thm:coh-pi11-wfepsilon0}
Fix an ordinal~$\alpha \leq \epsilon_0$.
$\RCA_0 + \COH + \BSig_2 + \WF(\alpha)$ is $\Pi_1^1$-conservative over $\RCA_0 + \BSig_2 + \WF(\alpha)$.
\end{theorem}

\begin{proof}
Immediate by \Cref{prop:coh-wf} and \Cref{theorem:problem-wf-cons}.
\end{proof}

\subsection{$\Pi^1_1$-conservation of $\WKL$}

We now turn to the proof of $\Pi^1_1$-conservation of~$\WKL_0$. As for~$\COH$, in the case~$\alpha = \omega$, we recover the well-known conservation theorem of H\'ajek~\cite{hajek1993interpretability} stating that $\RCA_0 + \BSig_2 + \WKL$ is a $\Pi^1_1$-conservative extension of~$\RCA_0 + \BSig_2$.

\begin{proposition}\label[proposition]{prop:wkl-wf}
Consider a countable model $\M = (M,S) \models \RCA_0 + \BSig_2 + \WF(\alpha)$ topped by a set $Y \in S$, where $\alpha \in M$ is an ordinal $\leq \epsilon_0$. Then for every $M$-infinite primitive $Y$-recursive tree $T \in S$, there exists $G \subseteq M$ such that 
\begin{enumerate}
    \item $G$ is a path in $T$ ;
    \item $Y' \geq_T (G \oplus Y)'$ ;
    \item $\M[G] \models \RCA_0 + \BSig_2 + \WF(\alpha)$
\end{enumerate}
\end{proposition}

\begin{proof}
We will build the path $G$ using forcing on trees. A condition will be a pair $(\sigma, T')$ with $T' \in S$ an $M$-infinite primitive $Y$-recursive subtree of $T$ with root $\sigma$ and we will have $(\sigma, T_1) \leq (\tau, T_2)$ if $\tau \prec \sigma$ and $T_1 \subseteq T_2$. 
\bigskip

\begin{definition}
Let $(\sigma, U)$ be a condition, and $\Phi_e$ be a Turing functional.
\begin{enumerate}
    \item $(\sigma,U) \Vdash \exists x \Phi_e^{G \oplus Y}(x)\downarrow$ if there is some $x,t < |\sigma|$ such that $\Phi_e^{\sigma \oplus Y}(x)[t] \downarrow$
    \item $(\sigma,U) \Vdash \forall x \Phi_e^{G \oplus Y}(x) \uparrow$ if for every $\tau \in U$ and every $x \in M$, $\Phi_e^{\tau \oplus Y}(x) \uparrow$.
\end{enumerate}
\end{definition}

Note that the relations $(\sigma, U) \Vdash \exists x \Phi^{G \oplus Y}_e(x)\downarrow$
and $(\sigma, U) \Vdash \forall x \Phi^{G \oplus Y}_e(x)\uparrow$ are $\Delta^0_0(Y)$ and $\Pi^0_1(Y)$, respectively.
As before, the definition above induces a forcing relation for all $\Sigma^0_1(G \oplus Y)$ and $\Pi^0_1(G \oplus Y)$ formulas.

For our construction, we want to satisfy the following requirements :
\begin{itemize}
    \item $\mathcal{R}_k$: $(G \oplus Y)' \uh_k$ is decided.
    \item $\mathcal{S}_{k}$: Let $\beta \in \alpha$ be the $<_{\epsilon_0}$-biggest element of index less than $k$, for every $e < k$, $\Phi_e^{G \oplus Y}$ is not an $M$-infinite $<_{\epsilon_0}$-decreasing sequence of elements smaller than~$\beta$.
    \item $\mathcal{T}_k$: $|\sigma| \geq k$.
\end{itemize}

We first prove a simple technical lemma saying that there exists arbitrarily large extendible nodes. This lemma even holds over~$\RCA_0^*$ and can be found in various places in the literature.

\begin{lemma}[Fernandes et al~\cite{fernandes2017analysis}]\label[lemma]{lem:picking-extendible-node-in-tree}
Let $\M = (M,S) \models \RCA_0^*$ be a model and $T \subseteq 2^{<M}$ be an $M$-infinite tree in~$\M$.
Then for every~$\ell \in M$, there is some~$\sigma \in T$ of length~$\ell$ such that $\{ \tau \in T : \tau \mbox{ is comparable with } \sigma \}$ is $M$-infinite.
\end{lemma}
\begin{proof}
Assume by contradiction that for every $\sigma \in 2^{\ell}$ the tree 
$$\{ \tau \in T : \tau \mbox{ is comparable with } \sigma \}$$ is $M$-bounded, this means that $\M \models \forall \sigma \in 2^\ell \exists b_{\sigma} \forall \tau \in 2^{b_{\sigma}}, \sigma \prec \tau \Rightarrow \tau \notin T$.

    The formula $\forall \tau \in 2^{b_{\sigma}}, \sigma \prec \tau \Rightarrow \tau \notin T$ is $\Delta_0^0(T)$ and therefore $\Delta_1^0$ in $\M$. Therefore, we can use $\BSig_1$ (which holds in $\RCA_0^*$) to get :

\[\M \models \exists b \forall \sigma \in 2^\ell \exists b_{\sigma} < b  \forall \tau \in 2^{b_{\sigma}}, \sigma \prec \tau \Rightarrow \tau \notin T\]

Which yields that $T$ is bounded by such a $b$, contradicting our assumption that $T$ is $M$-infinite. So there is some $\sigma \in T$ of length $\ell$ such that $\{ \tau \in T : \tau \mbox{ is comparable with } \sigma \}$ is $M$-infinite. 
\end{proof}

We will need the following two technical lemmas, which are Shore blocking arguments for satisfying the two kind of requirements, namely, deciding the jump and preserving $\WF(\alpha)$. 
In what follows, $(\tau, U) \Vdash (G \oplus Y)' \uh_k = \sigma'$ means that for every~$e < |\sigma'|$, if $\sigma'(e) = 1$ then $(\tau, U) \Vdash \Phi^{G \oplus Y}_e(e)\downarrow$ and if $\sigma'(e) = 0$, then $(\tau, U) \Vdash \Phi^{G \oplus Y}_e(e)\uparrow$. Thus, this relation is $\Pi^0_1(Y)$.

\begin{lemma}\label[lemma]{lem:wkl-wf-forcing-jump}
Let $(\sigma, T_1)$ a condition. For every $k \in M$, there exists some $M$-finite $\sigma'$ and an extension $(\tau, T_2) \leq (\sigma, T_1)$ forcing $(G \oplus Y)' \uh_k = \sigma'$.
\end{lemma}

\begin{proof}
Let $W$ be the set of all $\sigma' \in 2^k$ such that the tree $\{\rho \in T_1 | \sigma'(e) = 0 \to \Phi_e^{\rho \oplus Y}(e) \uparrow\}$ is $M$-infinite. $W$ is $\Pi_0^1(Y)$ and hence $M$-finite, and it is non-empty as it contains the string $1111\dots$.

We can consider $\sigma' \in W$ its lexicographically minimal element, and consider the corresponding tree $T_3$. For every $e < k$ such that $\sigma'(e) = 1$, the minimality of $\sigma'$ gives us that the set of $\tau \in T_3$ such that $\Phi_e^{\tau \oplus Y}(e) \uparrow$ is $M$-finite, so we have a level $h_e$ such that for every $\tau \in T_3 \cap 2^{h_e}$, $\Phi_e^{\tau \oplus Y}(e) \downarrow$. The set $\{e<k|\sigma'(e) = 1\}$ is $M$-finite, so we can consider an upperbound $h$ of all the $h_e$. Finally, by \Cref{lem:picking-extendible-node-in-tree}, there is a $\tau \in T_3 \cap 2^{h}$ such that $T_2 := \{\rho \in T_3 : \rho \textit{ is comparable with } \tau\}$ is $M$-infinite (since $T_3$ is $M$-infinite).

We claim that $(\tau, T_2)$ forces $(G \oplus Y)' \uh_k = \sigma'$. Fix some $e < k$.

Case (a): $\sigma'(e) = 0$. Then $\Phi_e^{\tau \oplus Y}(e) \downarrow$ for every $\tau \in T_2$ (since $T_2 \subseteq T_3$), hence for every $G$ compatible with $(\tau,T_2)$, so $e \notin (G \oplus Y)'$.

Case (b): $\sigma'(e) = 1$. The definition of $\tau$ ensure that $\Phi_e^{\tau \oplus Y}(e) \downarrow$, so $e \in (G \oplus Y)'$. 
\end{proof}

\begin{lemma}\label[lemma]{lem:wkl-wf-forcing-wf}
Let $(\sigma, T_1)$ be a condition. For every $k \in M$, letting~$\Gamma_k$ be the functional of \Cref{lem:shore-blocking-wf}, there is an extension $(\sigma, T_2) \leq (\sigma,T_1)$ forcing $\Gamma^{G \oplus Y}_k$ to be partial.
\end{lemma}

\begin{proof}
By twisting $\Gamma_k$, we can assume that for every~$e, n \in M$, if $\Gamma^{\sigma \oplus Y}_k(n)\downarrow$, then
\begin{itemize}
    \item $\Gamma^{\sigma \oplus Y}_k(m)\downarrow$ for every~$m < n$ ;
    \item $\Gamma^{\sigma \oplus Y}_k(0), \Gamma^{\sigma \oplus Y}_k(1), \dots, \Gamma^{\sigma \oplus Y}_k(n)$ is a strictly $<_{\epsilon_0}$-decreasing sequence of ordinals smaller than~$\beta \ntimes k$.
\end{itemize}
We have two cases.

Case 1: there exists some~$n \in M$ such that for every~$\ell \in M$, there is some~$\tau \in T_1$ of length~$\ell$ such that $\Gamma^{\tau \oplus Y}_k(n)\uparrow$. Let~$T_2 = \{ \tau \in T_1 : \Gamma^{\tau \oplus Y}_k(n)\uparrow \}$. Note that the set $T_2$ is primitive $Y$-recursive, as the set~$T_1$ and the predicate $\Gamma^{\sigma \oplus Y}_k(n)\uparrow$ are primitive $Y$-recursive. The set~$T_2$ is closed-downward, and by assumption, is $M$-infinite. Then $(\sigma, T_2) \leq (\sigma, T_1)$. Moreover, by construction, $(\sigma, T_2) \Vdash \Gamma_k^{G \oplus Y}(n)\uparrow$.

Case 2: for every~$n \in M$, there is some~$\ell_n \in M$ such that 
for every~$\tau \in T$ of length~$\ell_n$, $\Gamma^{\tau \oplus Y}_k(n)\downarrow$.

Then, for every~$n$, let $\alpha_n = \max\ \{ \Gamma^{\tau \oplus Y}_k(n) : \tau \in T_1 \wedge |\tau| = \ell_n \}$.
We claim that for every~$n \in M$, $\alpha_{n+1} <_{\epsilon_0} \alpha_n$.
Indeed, for every~$\tau \in T_1$ such that $|\tau| = \ell_{n+1}$, $\Gamma^{\tau \oplus Y}_k(n+1) <_{\epsilon_0} \Gamma^{(\tau \uh \ell_n) \oplus Y}_k$, so
$$
\max\ \{ \Gamma^{\tau \oplus Y}_k(n+1) : \tau \in T_1 \wedge |\tau| = \ell_{n+1} \}
<_{\epsilon_0} \max\ \{ \Gamma^{\tau \oplus Y}_k(n) : \tau \in T_1 \wedge |\tau| = \ell_n \}
$$
So $\M \mathbin{\cancel{\models}} \WF(\alpha \ntimes k)$. However, by \Cref{lem:wf-product}, since $\M \models \BSig_2 + \WF(\alpha)$, $\M \models \WF(\alpha \ntimes k)$. Contradiction.
\end{proof}

Note that the statement \qt{$(\sigma, T_1)$ forces $\Gamma_k^{G \oplus Y}(n)\uparrow$} is $\Pi^0_1(Y)$.
\bigskip

\textbf{Construction}. We will build a decreasing sequence $(\sigma_s, T_s)$ of conditions and then take $G$ for the union of the $\sigma_s$. We will also build an increasing sequence $(\sigma_s')$ such that $(G \oplus Y)'$ will be the union of the $\sigma_s'$. Initially we take $\sigma_0 = \sigma_0' = \epsilon$ and $T_0 = T$, and during the construction we will ensure that we have $|\sigma_s|,|\sigma'_s| \leq s$ and that the index describing $T_s$ is also less than $s$ at every stage. Each stage will be either of type $\mathcal{R}$, of type $\mathcal{S}$ or of type $\mathcal{T}$. The stage $0$ is of type $\mathcal{R}$.

Assume that $(\sigma_s, T_s)$ and $\sigma'_s$ are already defined. Let $s_0 < s$ be the latest stage at which we switched the stage type. We have three cases.

Case 1: $s$ is of type $\mathcal{R}$. If there exists some index $e \leq s$ such that the $e$-th primitive $Y$-recursive tree $T'$ is $M$-infinite and satisfy $(\sigma_s,T') \leq (\sigma_s, T_s)$ and is such that $(\sigma_s,T')$ forces $(G \oplus Y)' \uh_{s_0} = \sigma'$ for a $\sigma' \in 2^{s_0}$. Then let $\sigma_{s+1} = \sigma_s$, $T_{s+1} = T'$, $\sigma'_{s+1} = \sigma'$ and let $s+1$ be of type $\mathcal{S}$. Otherwise, the elements are left unchanged and we go to the next stage.

Case 2: $s$ is of type $\mathcal{S}$. If there exists some $e, n \leq s$ such that the $e$-th primitive $Y$-recursive tree $T'$ is $M$-infinite and $(\sigma_s, T')$ forces $\Gamma_{s_0}^{G \oplus Y}(n)\uparrow$, then let $\sigma_{s+1} = \sigma_s$, $T_{s+1} = T'$, $\sigma_{s+1}' = \sigma'_s$ and let $s+1$ be of type $\mathcal{T}$. Otherwise, the elements are left unchanged and we go to the next stage.

Case 3: $s$ is of type $\mathcal{T}$. If there exists some $\sigma \in 2^{s_0}$ such that the tree $\{\rho \in T_s | \rho \textit{ compatible with } \sigma\}$ is $M$-infinite and has index $\leq s$, then take this tree for $T_{s+1}$ and let $\sigma_{s+1} = \sigma$, $\sigma'_{s+1} = \sigma'_s$ and let $s+1$ be of type $\mathcal{R}$.


This completes the construction.
\bigskip

\textbf{Verification}. Since the size of $\sigma_s, \sigma'_s$ and the index of $T_s$ are bounded by $s$, there is a $\Delta_1^0(Y')$-formula $\phi(s)$ stating that the construction can be pursued up to stage $s$. Our construction implies that the set $\{s|\phi(s)\}$ is $\Delta_1^0(Y')$ and forms a cut, so by $\IDel_1(Y')$, the construction can be pursued at every stage. 

Let~$G = \bigcup_{s \in M} \sigma_s$.
One eventually switches from stages of type $\mathcal{R}$ to stages of type $\mathcal{S}$, by \Cref{lem:wkl-wf-forcing-jump}. The same occurs to stages $\mathcal{S}$ and $\mathcal{T}$ by \Cref{lem:wkl-wf-forcing-wf} and $\BSig_2$, respectively, so each type of stage changes $M$-infinitely often. Thus, $\{|\sigma_s| : s \in M\}$ and $\{|\sigma'_s| : s \in M \}$ are $M$-infinite, so $G$ is a path in $T$ and $Y' \geq_T (G \oplus Y)'$. In particular, $\M[G] \models \RCA_0 + \BSig_2$. 

Moreover, the stages $s_0$ where we force $\S_{s_0}$ are infinite, so for every $\beta \in \alpha$ and $e \in M$, there is a stage that will force $\Phi^{G \oplus Y}_e$ not to be an $M$-infinite $\alpha$-decreasing sequence of elements less than $\beta$. And thus $\Phi^{G \oplus Y}_e$ will not be an $M$-infinite $\alpha$-decreasing sequence for any $e$, so $\M[G] \models \WF(\alpha)$.

This completes the proof of \Cref{prop:wkl-wf}.
\end{proof}

\begin{theorem}\label[theorem]{thm:wkl-pi11-wfepsilon0}
Fix an ordinal $\alpha \leq \epsilon_0$.
$\RCA_0 + \WKL + \BSig_2 + \WF(\alpha)$ is $\Pi_1^1$-conservative over $\RCA_0 + \BSig_2 + \WF(\alpha)$.
\end{theorem}

\begin{proof}
Immediate by \Cref{prop:wkl-wf} and \Cref{theorem:problem-wf-cons}.
\end{proof}

\section{Conservation over $\BSig_2 + \WF(\epsilon_0)$}\label[section]{sect:bsig2-epsilon0}

The cases of~$\SEM$ and $\SADS$ are more complicated than those of~$\WKL$ and~$\COH$, as $\WF(\omega^{\alpha \times \omega})$ seems to be necessary to preserve $\WF(\alpha)$ where~$\alpha \leq \epsilon_0$. The ordinal $\epsilon_0$ is not a fixed point of $\alpha \mapsto \omega^{\alpha \times \omega}$. Fortunately, to preserve $\WF(\alpha)$, it is sufficient to preserve $\WF(\beta)$ for every $\beta < \alpha$. This is why, in \Cref{lem:shore-blocking-wf}, we considered the largest ordinal $\beta < \alpha$ of index at most~$k$, rather than $\alpha$ itself. Thus, we obtain that $\WF(\sup_{\beta < \alpha} \omega^{\beta \times \omega})$ suffices to preserve $\WF(\sup_{\beta < \alpha} \beta)$, and therefore that $\WF(\epsilon_0)$ suffices to preserve itself.


\subsection{$\Pi^1_1$-conservation of $\SEM$}

As mentioned in the introduction, the bounded monotone enumeration principle naturally shows up when formalizing the first-jump control of Ramsey's theorem for pairs. Within the decomposition of $\RT^2_2$ into $\COH + \SADS + \SEM$ of \Cref{prop:rt22-sem-sads-coh}, $\SEM$ is the only statement not known to be $\Pi^1_1$-conservative over~$\RCA_0 + \BSig_2$. It is therefore natural to expect to involve $\BME_*$ in the proof below. Since~$\BME_*$ is equivalent to~$\WF(\omega^\omega)$ and we suppose $\M \models \RCA_0 + \BSig_2 + \WF(\sup_{\beta < \alpha} \omega^{\beta \times \omega})$, we will be able to use~$\BME_*$.

\begin{proposition}\label[proposition]{prop:sem-wf}
Consider a countable model $\M = (M,S) \models \RCA_0 + \BSig_2 + \WF(\sup_{\beta < \alpha} \omega^{\beta \times \omega})$ topped by a set $Y \in S$, where~$\alpha \in M$ is a non-zero ordinal $\leq \epsilon_0$. Then for every infinite stable tournament $R$ in $S$ and every set~$P \gg Y'$ such that $\M[P] \models \RCA_0^*$, there exists $G \subseteq M$ such that  
\begin{enumerate}
    \item $G$ is an $M$-infinite transitive subtournament ;
    \item $P \geq_T (G \oplus Y)'$ ;
    \item $\M[G] \models \RCA_0 + \BSig_2 + \WF(\alpha)$.
\end{enumerate}
\end{proposition}

\begin{proof}

Fix a uniform enumeration of all $Y$-primitive recursive non-empty tree functionals $T_0, T_1, \dots$, 
that is, for every~$n \in M$ and every~$X$, $T_n^X$ is an $M$-infinite tree. Let
$$
\C = \{ \bigoplus_a X_a : \forall a \forall n\ X_{\langle a, n\rangle} \in [T_n^{Y \oplus X_0 \oplus \dots \oplus X_a}] \}
$$
The class $\C$ is $\Pi^0_1(Y)$ and non-empty. There exists a primitive $Y$-recursive tree whose infinite path are the elements of $\C$, so \Cref{prop:wkl-wf} gives us a set $\vec{X} = \bigoplus_a X_a \in \C$ such that $Y' \geq_T (\vec{X} \oplus Y)'$ and $\M[\vec{X}] \models \RCA_0 + \BSig_2 + \WF(\sup_{\beta < \alpha} \omega^{\beta \times \omega})$.

Furthermore, \Cref{prop:coh-wf} applied on $\M[\vec{X}]$ gives us a set $C \subseteq M$ satisfying  :
\begin{itemize}
    \item $C$ is $\vec{X}$-cohesive ;
    \item $P \geq_T (C \oplus \vec{X} \oplus Y)'$ ;
    \item $\M[\vec{X} \oplus C] \models \RCA_0 + \BSig_2 + \WF(\sup_{\beta < \alpha} \omega^{\beta \times \omega})$
\end{itemize}

Let $A_0(R) := \{x : \forall^{\infty} y R(x,y) \}$ and $A_1(R) := \{x : \forall^{\infty} y R(y,x)\}$. The stability of $R$ gives us $A_0(R) \sqcup A_1(R) = M$, so $A_0(R)$ and $A_1(R)$ are $\Delta_2^0(Y)$.

\begin{definition}
For two sets $F$ and $G$, the notation $F \to_R G$ denotes the formula $(\forall x \in F)(\forall y \in G)R(x,y)$.     
\end{definition}

\begin{definition}[Minimal interval] 
Let $a,b \in \M$, the interval $(a,b)$ is the set of all $x$ such that $R(a,x)$ and $R(x,b)$.

Let $F$ be an $M$-finite transitive subtournament of $R$, for $a,b \in F \cup \{-\infty, +\infty\}$ we say that $(a,b)$ is a minimal interval of $F$ is $F \cap (a,b) = \emptyset$, where we consider that $R(-\infty, x)$ and $R(x, +\infty)$ holds for every~$x \in \M$.
\end{definition}

\begin{definition}
A condition is a pair $(\sigma,a)$ where $\sigma$ is a $M$-finite binary string corresponding to an $R$-transitive subtournament, $a \in \M$ is such that $C \subseteq^* X_a$ and $X_a$ is included in a minimal interval of $\sigma$.
\end{definition}

Being a condition is a $\Delta_1^0((C \oplus X_a \oplus Y)')$-formula. Indeed, being a finite $R$-transitive subtournament is $Y$ computable, checking that $X_a$ is included in a minimal interval of~$\sigma$ is $(X_a \oplus Y)'$-computable, and checking if $C \subseteq^* X_{a}$ can be decided using $(C \oplus X_a)'$ (since $C$ is $\vec{X}$-cohesive, we know that $C \subseteq^* X_a$ or $C \subseteq^* \bar{X_a}$ and these events are $\Sigma_2^0(C \oplus X_a)$). Since $P \geq_T (C \oplus \vec{X} \oplus Y)'$, it is a $\Delta_1^0(P)$-formula.

\begin{definition}
A condition $(\tau,b)$ extends a condition $(\sigma,a)$ (and we write $(\tau,b) \leq (\sigma,a)$) if $\sigma \prec \tau$, $X_b \subseteq X_a$ and $\tau - \sigma \subseteq X_a$.
\end{definition}

For our construction, we want to satisfy the following requirements :
\begin{itemize}
    \item $\mathcal{R}_k$: $(G \oplus Y)' \uh_k$ is decided.
    \item $\mathcal{S}_{k}$: Let $\beta < \alpha$ be the $<_{\epsilon_0}$-biggest ordinal of index less than $k$, for every $e < k$, $\Phi_e^{G \oplus Y}$ is not an $M$-infinite $<_{\epsilon_0}$-decreasing sequence of elements smaller than $\beta$.
    \item $\mathcal{T}_k$: $\card G \geq k$.
\end{itemize}

We will need two technical lemmas, which are Shore blocking arguments for satisfying the two kinds of requirements, namely, deciding the jump and preserving $\WF(\alpha)$. Before this, we introduce some basic combinatorial concepts which will be useful for both lemmas. A $\Sigma^0_1$ formula $\varphi(G)$ is in \emph{normal form} if $\varphi(G) \equiv \exists i \psi(G \uh i)$ for a $\Delta^0_0$ formula $\psi$.

\begin{definition}
Fix a condition $(\sigma, a)$.
\begin{enumerate}
    \item[(1)] A \emph{$(\sigma, a)$-node} is an $M$-finite sequence of the form
$$\nu = \langle (i_0, \rho_0), (i_1, \rho_1), \dots (i_{q-1}, \rho_{q-1}) \rangle$$ such that for every~$s < q$, $i_s < 2$, $\rho_s \subseteq X_a$ is a non-empty $M$-finite $R$-transitive set with $\min \rho_s > \max \rho_{s-1}$ (letting $\rho_{-1} = \sigma$), and for every~$s < t < q$, if $i_s = 0$ then $\rho_s \to_R \rho_t$, and if $i_s = 1$ then $\rho_t \to_R \rho_s$.
We then write $$\pi(\nu) = \sigma \cup \rho_0 \cup \dots \cup \rho_{q-1}$$
    \item[(2)] This $(\sigma, a)$-node is \emph{valid} if for every $s < q$, $\rho_s \subseteq A_{i_s}(R)$
    \item[(3)] A $(\sigma, a)$-node $\nu$ \emph{satisfies} a $\Sigma^0_1(\M)$-formula $\varphi(G, x) \equiv \exists i \psi(G \uh i, x)$  if $\psi(\pi(\nu), |\nu|-1)$ holds.
\end{enumerate}
\end{definition}

\begin{lemma}\label[lemma]{lem:sem-wf-nodes-transitive}
Fix a $(\sigma, a)$-node $\nu$. Then $\pi(\nu)$ is $R$-transitive.
\end{lemma}
\begin{proof}
Say $\nu = \langle (i_0, \rho_0), (i_1, \rho_1), \dots (i_{q-1}, \rho_{q-1}) \rangle$.
We prove by $\Sigma^0_0$-induction that for every~$\ell < t$, $\rho_\ell$ is an $R$-transitive tournament included in a minimal interval of $\sigma \cup \rho_0 \cup \dots \cup \rho_{\ell-1}$ and that $\sigma \cup \rho_0 \cup \dots \cup \rho_{\ell}$ is $R$-transitive.

This is true for $\ell = 0$ since $(\sigma, a)$ is a condition. And assuming the property to be true at a rank $\ell < t-1$, the construction of $\rho_{\ell +1}$ ensures that it is included in the minimal interval of $(\sigma,a)$ and that $\rho_{\ell + 1} \to_R \rho_{j}$ if $i_j = 1$ and $\rho_j \to_R \rho_{\ell + 1}$ if $i_j = 0$, so no elements of $\sigma \cup \rho_0 \cup \dots \cup \rho_{\ell}$ is between two elements of $\rho_{\ell+1}$. Since $\sigma \cup \rho_0 \cup \dots \cup \rho_{\ell}$ is $R$-transitive (by the induction hypothesis), this means that $\rho_{\ell +1}$ is included in one of its minimal intervals, and since it is itself $R$-transitive, $\sigma \cup \rho_0 \cup \dots \cup \rho_{\ell} \cup \rho_{\ell +1}$ is $R$-transitive.
\end{proof}

\begin{lemma}\label[lemma]{lem:sem-wf-valid-leaf-force-positive}
Fix a valid $(\sigma, a)$-node~$\nu$.
Then $(\pi(\nu), b)$ is an extension of~$(\sigma, a)$ for some~$b \in M$.
\end{lemma}
\begin{proof}
Say $\nu = \langle (i_0, \rho_0), (i_1, \rho_1), \dots (i_{q-1}, \rho_{q-1}) \rangle$.
By definition of validity and stability of~$R$, for every~$s < q$, if $i_s = 0$ then $\forall x \in \rho_s \forall^\infty y \in X_a$, $R(x, y)$, and if $i_s = 1$ then $\forall x \in \rho_s \forall^\infty y \in X_a$, $R(y, x)$.
By $\BSig_2$, there exists some~$r \in M$ such that for every~$s < q$, if $i_s = 0$ then $\rho_s \to_R X_a \setminus \{0, \dots, r\}$, and if $i_\ell = 1$ then $X_a \setminus \{0, \dots, r\} \to_R \rho_s$. Let~$b \in M$ be such that $X_b = X_a \setminus \{0, \dots, r\}$ and let~$\tau = \sigma \cup \rho_0 \cup \dots \cup \rho_{t-1}$.
By \Cref{lem:sem-wf-nodes-transitive}, $\tau$ is $R$-transitive and by construction, $X_b \subseteq X_a$ is included in a minimal interval of $\tau$. Thus $(\tau, b)$ is an extension of $(\sigma, a)$.
\end{proof}

\begin{definition}
Fix a condition $(\sigma, a)$.
\begin{enumerate}
    \item[(1)] A \emph{$(\sigma, a)$-tree} is a set of $(\sigma, a)$-nodes closed downard under the prefix relation.
    \item[(2)] A $(\sigma, a)$-tree $E$ is \emph{universal} if for every~$(\sigma,a)$-node $\nu$
    which is not a leaf, there is some~$r \in M$ such that for every $M$-finite 2-partition $A_0 \sqcup A_1 = \{0, \dots, r-1\}$, there is a side~$i < 2$ and an $M$-finite $R$-transitive $\rho \subseteq A_i \cap X_a$ such that $\nu \cdot \langle i, \rho \rangle \in E$.
    \item[(3)] A $(\sigma, a)$-tree $E$ satisfies a $\Sigma^0_1(\M)$ formula if all its nodes satisfies it.
\end{enumerate}
\end{definition}

\begin{lemma}\label[lemma]{lem:sem-wf-exists-maximal-enumeration}
For every~$\Sigma^0_1(\M)$ formula $\varphi(G, x)$ and every condition $(\sigma, a)$, there is a monotone enumeration of a maximal (for inclusion) universal $(\sigma, a)$-tree~$E$ satisfying $\varphi(G, x)$.
\end{lemma} 
\begin{proof}
At stage $0$ enumerate the empty sequence $\langle \rangle$.
At stage $s$, for every leaf $\nu \in E$, if for every 2-partition $A_0 \sqcup A_1 = \{0,\dots,s\}$, there is some~$i < 2$ and some~$\rho \subseteq A_i$ for which 
\begin{enumerate}
    \item[(i)] $\varphi(\pi(\nu) \cup \rho, |\nu|)$ holds
    \item[(ii)] $\nu \cdot \langle (i, \rho) \rangle$ is a $(\sigma, a)$-node.
\end{enumerate}
then enumerate $\nu \cdot \langle (i, \rho) \rangle$ for all such partitions, and go to the next stage.
If nothing is enumerated, simply go to the next stage.
\end{proof}

A maximal universal $(\sigma, a)$-tree is not necessarily $M$-finite.

\begin{lemma}\label[lemma]{lem:sem-wf-valid-leaf}
Every $M$-finite universal $(\sigma, a)$-tree $E$ has a valid leaf.
\end{lemma}
\begin{proof}
We prove it by $\IDel_2$ induction. 
The root $\langle \rangle$ is vacuously valid. Assume  $\nu \in E$ is a valid non-leaf.
By universality of~$E$, there is some~$r \in M$ such that for every $M$-finite 2-partition $A_0 \sqcup A_1 = \{0, \dots, r-1\}$, there is a side~$i < 2$ and an $M$-finite $R$-transitive $\rho \subseteq A_i \cap X_a$ such that $\nu \cdot \langle i, \rho \rangle \in E$. By $\BCDe_2(Y)$ (which is equivalent to $\BSig_2(Y)$), the following partition is $M$-finite:
$$A_0(R) \uh_r \sqcup A_1(R) \uh_r = \{0,\dots, r-1\}$$
Thus, there exists some~$i < 2$ and some~$\rho \subseteq A_i(R)$ such that $\nu \cdot \langle i, \rho \rangle \in E$. By choice of~$i$ and $\rho$, the node $\nu \cdot \langle i, \rho \rangle$ is valid.
\end{proof}

\begin{lemma}\label[lemma]{lem:sem-wf-finite-tree-forcing}
Let $(\sigma,a)$ be a condition, $\varphi(G, x)$ be a $\Sigma^0_1(\M)$ formula, and $E$ be an $M$-finite maximal universal $(\sigma, a)$-tree satisfying $\varphi(G, x)$.
Then there is an extension $(\tau, c) \leq (\sigma, a)$ and some~$q \in M$ such that
$(\tau, c) \Vdash \varphi(G, q-1)$ and $(\tau, c) \Vdash \neg \varphi(G, q)$.
\end{lemma}
\begin{proof}
By \Cref{lem:sem-wf-valid-leaf}, there is a valid leaf  $\nu \in E$.
Since $E$ satisfies $\varphi(G, x)$, $\varphi(\pi(\nu), |\nu|-1)$ holds.
By \Cref{lem:sem-wf-valid-leaf-force-positive}, there is some~$b \in M$ such that $(\pi(\nu), b) \leq (\sigma, a)$.

By maximality of the tree, for every~$s \in M$, there exists a 2-partition $A_0 \sqcup A_1 = \{0,\dots,s\}$, such that for every side~$i < 2$ and every $M$-finite $\rho \subseteq A_i \cap X_a$ such that
$\nu \cdot \langle i, \rho \rangle$ is a $(\sigma, a)$-node, then $\varphi(\pi(\nu) \cup \rho, |\nu|)$ does not hold.

Note that by choice of~$X_b$, for every $M$-finite $R$-transitive $\rho \subseteq A_i \cap X_b$, $\nu \cdot \langle i, \rho \rangle$ is a $(\sigma, a)$-node. Thus, for every~$s \in M$, there exists a 2-partition $A_0 \sqcup A_1 = \{0,\dots,s\}$, such that the following property holds:
\begin{quote}
(P): For every side~$i < 2$ and every $M$-finite $R$-transitive $\rho \subseteq A_i \cap X_b$, $\varphi(\pi(\nu) \cup \rho, |\nu|)$ does not hold.
\end{quote}

By compactness, there exists a 2-partition $A_0 \sqcup A_1 = X_b$ such that (P) holds. Moreover, the class $\C$ of all such partitions is $\Pi^0_1(Y \oplus X_b)$ and non-empty. By definition of $\vec{X}$, there exists some~$c,d \in M$ such that $X_c, X_d$ is a 2-partition of~$X_b$ in~$\C$. Moreover, since~$G \subseteq^* X_b$, either~$G \subseteq^* X_c$ or $G \subseteq^* X_d$. Since they play a symmetric role, say $G \subseteq^* X_c$. Then $(\tau, c)$ is an extension of $(\tau, b)$ forcing $\varphi(G, |\nu|-1)$ and $\neg \varphi(G, |\nu|)$.
\end{proof}

We are now ready to prove our two Shore blocking lemmas.

\begin{lemma}\label[lemma]{lem:sem-wf-forcing-jump}
Let $(\sigma,a)$ be a condition. For every $k \in M$, there exists some $M$-finite $\sigma' \in 2^k$ and an extension $(\tau,b) \leq (\sigma,a)$ such that $(\tau,b) \Vdash (G \oplus Y)' \uh k = \sigma'$    
\end{lemma}
\begin{proof}
Let $\varphi(G, x)$ be the $\Sigma^0_1(Y)$ formula which holds if there exists some~$e_0 < \dots < e_{x-1} < k$ such that for every~$s < x$, $\Phi^{G \oplus Y}_{e_s}(e_s)\downarrow$.

By \Cref{lem:sem-wf-exists-maximal-enumeration}, there exists a monotone enumeration of a maximal universal $(\sigma, a)$-tree~$E$ satisfying $\varphi(G, x)$.
Moreover, since~$\varphi(G, k+1)$ never holds, the enumeration $E$ is $k$-bounded.
Since $\alpha > 0$, $\M \models \WF(\omega^\omega)$, hence $\M \models \BME_*$, so $E$ is $M$-finite.
By \Cref{lem:sem-wf-finite-tree-forcing}, there is an extension $(\tau, c) \leq (\sigma, a)$ and some~$q \in M$ such that $(\tau, c) \Vdash \varphi(G, q-1)$ and $(\tau, c) \Vdash \neg \varphi(G, q)$.
Let~$\sigma' = \{ e < k : \Phi^{\tau \oplus Y}_e(e)\downarrow \}$. Then $(\tau,b) \Vdash (G \oplus Y)' \uh k = \sigma'$.
\end{proof}

\begin{lemma}\label[lemma]{lem:sem-wf-forcing-wf}
Let $(\sigma, a)$ be a condition. For every $k \in M$, letting~$\Gamma_k$ be the functional of \Cref{lem:shore-blocking-wf} for the order $\alpha$, there is an extension $(\tau, b) \leq (\sigma,a)$ forcing $\Gamma^{G \oplus Y}_k$ to be partial.
\end{lemma}

\begin{proof}
By twisting $\Gamma_k$, we can assume that for every~$e, n \in M$, if $\Gamma^{\sigma \oplus Y}_k(n)\downarrow$, then
\begin{itemize}
    \item $\Gamma^{\sigma \oplus Y}_k(m)\downarrow$ for every~$m < n$ ;
    \item $\Gamma^{\sigma \oplus Y}_k(0), \Gamma^{\sigma \oplus Y}_k(1), \dots, \Gamma^{\sigma \oplus Y}_k(m)$ is a strictly $<_{\epsilon_0}$-decreasing sequence of elements smaller than~$\beta \ntimes k$.
\end{itemize}

Let $\varphi(G, x)$ be the $\Sigma^0_1(Y)$ formula which holds if $\Gamma^{G \oplus Y}_k(x)\downarrow$.
By \Cref{lem:sem-wf-exists-maximal-enumeration}, there exists a monotone enumeration of a maximal universal $(\sigma, a)$-tree~$E$ satisfying $\varphi(G, x)$. We have two cases.

Case 1: $E$ is $M$-finite. Then by \Cref{lem:sem-wf-finite-tree-forcing}, there is an extension $(\tau, c) \leq (\sigma, a)$ and some~$q \in M$ such that $(\tau, c) \Vdash \varphi(G, q-1)$ and $(\tau, c) \Vdash \neg \varphi(G, q)$. In particular, $(\tau, c) \Vdash \Gamma^{G \oplus Y}_k(x)\uparrow$.

Case 2: $E$ is not $M$-finite. Then there exists $M$-infinitely many stages~$s \in M$ such that $E[s] \neq E[s+1]$.
By speeding up the enumeration, we can assume without loss of generality that $E[s+1] \neq E[s] \neq \{\langle\rangle\}$ for every~$s \in M$.

For every~$s \in M$ and $\nu \in E[s]$, let
$$\zeta_s(\nu) = \left\{ \begin{array}{ll}
    \omega^{\Gamma_k^{\pi(\nu) \oplus Y}(|\nu|-1)} & \mbox{ if } \nu \mbox{ is a leaf of } E[s]\\
    \nsum_{\mu \sqsupset_{E[s]} \nu} \zeta_s(\mu) & \mbox{ otherwise }
\end{array}\right.
$$
Finally, let $\zeta_s = \zeta_s(\langle\rangle)$.
\bigskip

\textbf{Claim 1}: \emph{$\zeta_{s+1} <_{\epsilon_0} \zeta_s$ for all $s \in M$.}
Since $E[s+1] \neq E[s]$, there is some leaf~$\nu \in E[s]$ which is not a leaf in~$E[s+1]$.
By assumption on $\Gamma_k$, for every~$\mu \sqsupseteq_{E[s+1]} \nu$,
$$
\Gamma_k^{\pi(\mu) \oplus Y}(|\mu|-1) <_{\epsilon_0} \Gamma_k^{\pi(\nu) \oplus Y}(|\nu|-1)
$$
Hence
$$
\nsum_{\mu \sqsupset_{E[s+1]} \nu} \omega^{\Gamma_k^{\pi(\mu) \oplus Y}(|\mu|-1)} <_{\epsilon_0} \omega^{\Gamma_k^{\pi(\nu) \oplus Y}(|\nu|-1)}
$$
So $\zeta_{s+1}(\nu) <_{\epsilon_0} \zeta_s(\nu)$. By $\Delta^0_0$ induction along the ancestors of~$\nu$ up to the root, we obtain $\zeta_{s+1}(\langle\rangle) <_{\epsilon_0} \zeta_s(\langle\rangle)$. This concludes our proof of Claim 1.\\

\textbf{Claim 2}: \emph{$\zeta_s <_{\epsilon_0} \omega^{\beta \times \omega}$ for all $s \in M$.} For all $s \in M$, $\zeta_s \leq_{\epsilon_0} \zeta_0 \leq_{\epsilon_0} \omega^{\beta \ntimes k} <_{\epsilon_0} \omega^{\beta \times \omega}$. 

Then Claim 1 and Claim 2 together contradict the fact that $\M \models \WF(\sup_{\beta < \alpha} \omega^{\beta \times \omega})$.

\end{proof}

\textbf{Construction} We will build a decreasing sequence $(\sigma_s,a_s)$ of conditions and $G$ to be the union of the $\sigma_s$. We will also build an increasing sequence $(\sigma_s')$ such that $G'$ will be the union of the $\sigma_s'$. Initially we take $\sigma_0 = \sigma_0' = \epsilon$ and $a_0$ an index corresponding to $M$, and during the construction we will ensure that we have $|\sigma_s|,|\sigma'_s| \leq s$ and that the index $a_s$ is also less than $s$ at every stage. Each stage will be either of type $\mathcal{R}$, of type $\mathcal{S}$ or of type $\mathcal{T}$. The stage $0$ is of type $\mathcal{R}$.

Assume that $(\sigma_s, a_s)$ and $\sigma'_s$ are already defined. Let $s_0 < s$ be the latest stage at which we switched the stage type. We have three cases.

Case 1: $s$ is of type $\mathcal{R}$. If there exists some condition $(\sigma,a) \leq (\sigma_s,a_s)$ with $|\sigma| \leq s$ and $a \leq s$ such that $(\sigma,a) \Vdash (G \oplus Y)' \uh_{s_0} = \sigma'$ for a $\sigma' \in 2^{s_0}$. Then let $\sigma_{s+1} = \sigma$, $a_{s+1} = a$, $\sigma'_{s+1} = \sigma'$ and let $s+1$ be of type $\mathcal{S}$. Otherwise, the elements are left unchanged and we go to the next stage.

Case 2: $s$ is of type $\mathcal{S}$. If there exists some condition $(\sigma,a) \leq (\sigma_s,a_s)$ with $|\sigma| \leq s$ and $a \leq s$ such that $(\sigma, a)$ forces $\Gamma_{s_0}^{G \oplus Y}(m)\uparrow$ for some~$m \leq s$, then let $\sigma_{s+1} = \sigma$, $a_{s+1} = a'$, $\sigma_{s+1}' = \sigma'_s$ and let $s+1$ be of type $\mathcal{T}$. Otherwise, the elements are left unchanged and we go to the next stage.

Case 3: $s$ is of type $\mathcal{T}$. If there is some condition $(\sigma,a) \leq (\sigma_s,a_s)$ with $\sigma - \sigma_s$ non-empty, $|\sigma| \leq s$ and $a \leq s$, then we let $\sigma_{s+1} = \sigma$, $a_{s+1} = a$ and $\sigma_{s+1}' = \sigma_s'$. Otherwise, the elements are left unchanged and we go to the next stage.

This completes the construction.
\bigskip

\textbf{Verification}. Since the size of $\sigma_s, \sigma'_s$ and the index $a_s$ are bounded by $s$, there is a $\Delta_1^0(P \oplus Y')$-formula $\phi(s)$ stating that the construction can be pursued up to stage $s$. Our construction implies that the set $\{s|\phi(s)\}$ is $\Delta^0_1(P)$ and a cut, so by $\IDel_1(P)$, the construction can be pursued at every stage.

Let~$G = \bigcup_{s \in M} \sigma_s$.
By \Cref{lem:sem-wf-forcing-jump}, \Cref{lem:sem-wf-forcing-wf}, every time we enter a stage of type $\R$ or $\S$ we eventually leave it. Also, for every condition $(\sigma_s, a_s)$, we can add an element of $X_{a_s}$ to $\sigma_s$ and then extract an infinite minimal interval from $X_{a_s}$ using $P$, so every time we enter a stage of type $\T$, we eventually leave it. So, each type of stage changes $M$-infinitely often. Thus, $\{|\sigma'_s| : s \in M \}$ is $M$-infinite, so $P \geq_T (G \oplus Y)'$. In particular, $\M[G] \models \RCA_0 + \BSig_2$. 

Moreover, the stages $s_0$ where we force $\S_{s_0}$ are infinite, so for every $\beta \in \alpha$ and $e \in M$, there is a stage that will force $\Phi^{G \oplus Y}_e$ not to be an $M$-infinite $\epsilon_0$-decreasing sequence of elements less than $\beta$. And thus $\Phi^{G \oplus Y}_e$ will not be an $M$-infinite $\alpha$-decreasing sequence for any $e$, so $\M[G] \models \WF(\alpha)$.

The stages of type $\T$ happens $M$-infinitely often, so the final set $G$ will be $M$-infinite. This completes the proof of \Cref{prop:sem-wf}.

\end{proof}

\begin{corollary}\label[corollary]{cor:sem-wf}
Consider a countable model $\M = (M,S) \models \RCA_0 + \BSig_2 + \WF(\sup_{\beta < \alpha} \omega^{\beta \times \omega})$ topped by a set $Y \in S$, where~$\alpha \in M$ is an ordinal $\leq \epsilon_0$. Then for every infinite stable tournament $R$ in $S$, there exists $G \subseteq M$ such that  
\begin{enumerate}
    \item $G$ is an $M$-infinite transitive subtournament ;
    \item $\M[G] \models \RCA_0 + \BSig_2 + \WF(\alpha)$.
\end{enumerate}
\end{corollary}
\begin{proof}
By \Cref{thm:wkl-rca0s}, there exists a set~$P \subseteq M$ such that $P \gg Y'$ and $\M[P] \models \RCA_0^*$.
The conclusion follows from~\Cref{prop:sem-wf}.
\end{proof}

\begin{theorem}\label[theorem]{thm:sem-pi11-wfepsilon0}
$\RCA_0 + \BSig_2 + \WF(\epsilon_0) + \SEM$ is $\Pi_1^1$-conservative over $\RCA_0 + \BSig_2 + \WF(\epsilon_0)$.
\end{theorem}

\begin{proof}
Immediate by \Cref{cor:sem-wf} and \Cref{theorem:problem-wf-cons}.
\end{proof}







\subsection{$\Pi^1_1$-conservation of $\SADS$}

We now turn to the proof that $\RCA_0 + \BSig_2 + \WF(\epsilon_0) + \SADS$ is a $\Pi^1_1$-conservative extension of~$\RCA_0 + \BSig_2 + \WF(\epsilon_0)$.
There exists two main constructions of solutions to instances of~$\SADS$: an asymmetric one, which assumes that there is no infinite computable descending sequence, and constructs an ascending one, as in Hirschfeldt and Shore~\cite[Theorem 2.11]{Hirschfeldt2007CombinatorialPW}, and a symmetric one using split pairs, as in Lerman, Solomon and Towsner~\cite{lerman2013separating} or Patey~\cite[Theorem 26]{patey2018partial}. Our proof follows the second construction. 

\begin{proposition}\label[proposition]{prop:sads-wf}
Consider a countable model $\M = (M,S) \models \RCA_0 + \BSig_2 + \WF(\sup_{\beta < \alpha} \omega^{\beta \times \omega})$ topped by a set $Y \in S$, where $\alpha \in M$ is an ordinal $\leq \epsilon_0$. Then for every linear order $L = (M, \leq_L) \in S$ of order type $\omega + \omega^*$, there exists $G \subseteq M$ such that  
\begin{enumerate}
    \item $G$ is an $M$-infinite ascending or descending sequence for $\leq_L$;
    \item $Y' \geq_T (G \oplus Y)'$ ;
    \item $\M[G] \models \RCA_0 + \BSig_2 + \WF(\alpha)$.
\end{enumerate}
\end{proposition}

\begin{proof}
Let $L = (M, \leq_L)$ be such an order. If $\M$ already contains an $M$-unbounded ascending or descending sequence for $\leq_L$, then we are done, so assume otherwise. Let $U \subseteq M$ and $V \subseteq M$ be the sets of elements that have a finite amount of predecessors and successors, respectively. By assumption, $U \sqcup V = M$, and $U$ and $V$ are both $\Delta^0_2(Y)$.

A condition is a pair $(\sigma, \tau)$ with $\sigma$ an $M$-finite $L$-ascending sequence included in $U$ and $\tau$ an $M$-finite $L$-descending sequence included in $V$. We have $(\sigma_2, \tau_2) \leq (\sigma_1, \tau_1)$ if $\sigma_2 \succeq \sigma_1$ and $\tau_2 \succeq \tau_1$.

Being a condition is a $\Delta_2^0(Y)$-predicate, while the extension relation is $\Delta_1^0$. We define a weaker notion of pair which is $\Delta^0_1(Y)$:
\begin{definition}
A \emph{split pair} is a pair $(\sigma, \tau)$ of strings such that $\sigma$ is an $M$-finite $L$-ascending sequence, $\tau$ is an $M$-finite $L$-descending sequence, and $\max_L \sigma <_L \min_L \tau$.
\end{definition}

\begin{definition}
Let $(\sigma_0, \sigma_1)$ be a condition, $e, x \in M$ and $i < 2$, we write:
\begin{itemize}
    \item $(\sigma_0, \sigma_1) \Vdash_i \Phi_e^{G \oplus Y}(x) \downarrow$ if $\Phi_e^{\sigma_i \oplus Y}(x) \downarrow$. 
    \item $(\sigma_0, \sigma_1) \Vdash_i \Phi_e^{G \oplus Y}(x) \uparrow$ if $\forall \sigma'_i \succeq \sigma_i$ such that $(\sigma'_i, \sigma_{1-i})$ is a split pair, $\Phi_e^{\sigma'_i \oplus Y}(x)\uparrow$.
\end{itemize}   
\end{definition}

We want our construction to satisfy for some~$i < 2$ the following requirements for every~$k \in M$:
\begin{itemize}
    \item $\mathcal{R}^i_k$: $(G_i \oplus Y)' \uh_k$ is decided.
    \item $\mathcal{S}^i_{k}$: Let $\beta < \alpha$ be the $<_{\epsilon_0}$-biggest ordinal of index less than $k$, for every $e < k$, $\Phi_e^{G_i \oplus Y}$ is not an $M$-infinite $<_{\epsilon_0}$-decreasing sequence of elements smaller than $\beta$.
    \item $\mathcal{T}^i_k$: $\card G_i \geq k$.
\end{itemize}
For this, we will use a pairing argument, and ensure that for every~$k \in M$, the following requirements are met: $\R^0_k \vee \R^1_k$, $\R^0_k \vee \S^1_k$, $\S^0_k \vee \R^1_k$, $\S^0_k \vee \S^1_k$, and $\T^0_k \wedge \T^1_k$.

The core property of split pairs $(\sigma, \tau)$ comes from the fact that if $\sigma \not \subseteq U$, then $\tau \subseteq V$. In particular, if $(\sigma, \tau)$ is a condition and $(\sigma', \tau')$ is a split pair such that $\sigma \preceq \sigma'$ and $\tau \preceq \tau'$, then either $(\sigma', \tau)$, or $(\sigma, \tau')$ is a valid condition extending $(\sigma, \tau)$.

We introduce a forcing question~\footnote{The proof of \Cref{lem:sads-wf-finite-tree-forcing} was flawed in the published version of this article at Transactions of the American Mathematical Society, 378 (2025), no.3, pp. 2157--2186. We fixed it here by introducing the forcing question and proving \Cref{lem:sads-forcing-question-sigma01}.} which thanks to \Cref{lem:sads-forcing-question-sigma01}, allows us to $Y'$-computably determine a side $i < 2$ and construct extensions forcing either a given property or its negation on side $i$.

\begin{definition}
    Let $(\sigma, \tau)$ be a condition and $e_0, e_1, x_0, x_1 \in M$. We write $(\sigma, \tau) \qvdash \Phi_{e_0}^{G_0 \oplus Y}(x_0) \downarrow \vee \Phi_{e_1}^{G_1 \oplus Y}(x_1) \downarrow$ if there exists some split pair $(\sigma', \tau')$ such that $\sigma \preceq \sigma'$, $\tau \preceq \tau'$, $\Phi_{e_0}^{\sigma' \oplus Y}(x_0) \downarrow$ and $\Phi_{e_1}^{\tau' \oplus Y}(x_1) \downarrow$. 
\end{definition}

The relation $(\sigma, \tau) \qvdash \Phi_{e_0}^{G_0 \oplus Y}(x_0) \downarrow \vee \Phi_{e_1}^{G_1 \oplus Y}(x_1) \downarrow$ is $\Sigma^0_1(Y)$. 

\begin{lemma}\label[lemma]{lem:sads-forcing-question-sigma01}
    Let $(\sigma_0, \tau_0)$ be a condition, and let $e_0, e_1, x_0, x_1 \in M$. Then:
    \begin{itemize}
        \item If $(\sigma_0, \tau_0) \qvdash \Phi_{e_0}^{G_0 \oplus Y}(x_0) \downarrow \vee \Phi_{e_1}^{G_1 \oplus Y}(x_1) \downarrow$, then there exist an extension $(\sigma_1, \tau_1) \leq (\sigma_0, \tau_0)$ and some $i < 2$ such that $(\sigma_1, \tau_1) \Vdash_i \Phi_{e_i}^{G_i \oplus Y}(x_i) \downarrow$.
        \item If $(\sigma_0, \tau_0) \nqvdash \Phi_{e_0}^{G_0 \oplus Y}(x_0) \downarrow \vee \Phi_{e_1}^{G_1 \oplus Y}(x_1) \downarrow$, then there exist an extension $(\sigma_1, \tau_1) \leq (\sigma_0, \tau_0)$ and some $i < 2$ such that $(\sigma_1, \tau_1) \Vdash_i \Phi_{e_i}^{G_i \oplus Y}(x_i) \uparrow$.
    \end{itemize}
    Furthermore, deciding which case holds and computing the extension $(\sigma_1, \tau_1)$ can be done $Y'$-computably.
\end{lemma}

\begin{proof}
Assume first that $(\sigma_0, \tau_0) \qvdash \Phi_{e_0}^{G_0 \oplus Y}(x_0) \downarrow \vee \Phi_{e_1}^{G_1 \oplus Y}(x_1) \downarrow$. Then there exists a split pair $(\sigma_2, \tau_2)$ such that $\sigma_0 \preceq \sigma_2$, $\tau_0 \preceq \tau_2$, $\Phi_{e_0}^{\sigma_2 \oplus Y}(x_0)\downarrow$ and $\Phi_{e_1}^{\tau_2 \oplus Y}(x_1)\downarrow$. Using $Y$, we can find such a split pair. Since $(\sigma_2, \tau_2)$ is a split pair, either $\sigma_2 \subseteq U$ or $\tau_2 \subseteq V$, and deciding which case holds can be done $Y'$-computably. If $\sigma_2 \subseteq U$, then letting $\sigma_1 := \sigma_2$ and $\tau_1 = \tau_0$ yields an extension $(\sigma_1, \tau_1) \leq (\sigma_0, \tau_0)$ such that $(\sigma_1, \tau_1) \Vdash_0 \Phi_{e_0}^{G_0 \oplus Y}(x_0) \downarrow$. The case $\tau_2 \subseteq V$ is symmetrical.  \\

Now suppose that $(\sigma_ 0, \tau_ 0) \nqvdash \Phi_{e_0}^{G_0 \oplus Y}(x_0) \downarrow \vee \Phi_{e_1}^{G_1 \oplus Y}(x_1) \downarrow$. Consider the program which runs through all $M$-finite $L$-ascending sequences $\sigma_2$ such that $\sigma_ 0 \preceq \sigma_2$, $\max_L \sigma_2 <_L \min_L \tau_0$ and $\Phi_{e_0}^{\sigma_2 \oplus Y}(x_0) \downarrow$. Whenever a new $\sigma_2$ is found whose value of $\max_L \sigma_2$ is smaller than all previously found ones, enumerate $\max_L \sigma_2$. By construction, this program yields a $Y$-computable $L$-descending sequence. By our assumption, $\M$ does not contain any solution to the given instance of $\SADS$, this enumeration must therefore be finite. Hence, either no such $\sigma_2$ is enumerated, in which case $(\sigma_ 0, \tau_ 0) \Vdash_0 \Phi_{e_0}^{G_0 \oplus Y}(x_0) \uparrow$ and we are done, or else we can $Y'$-computably find some $\sigma_2$ such that $\sigma_0 \preceq \sigma_2$, $\max_L \sigma_2 <_L \min_L \tau_0$ and $\Phi_{e_0}^{\sigma_2 \oplus Y}(x_0) \downarrow$ and for which $\max_L \sigma_2$ is minimal. 

Similarly, either we already have $(\sigma_ 0, \tau_ 0) \Vdash_1 \Phi_{e_1}^{G_1 \oplus Y}(x_1) \uparrow$ or we can $Y'$-computably find some $M$-finite $L$-descending sequence $\tau_2$ such that $\tau_0 \preceq \tau_2$, $\min_L \tau_2 >_L \max_L \sigma_0$ and $\Phi_{e_1}^{\tau_2 \oplus Y}(x_1) \downarrow$ and such that $\min_L \tau_2$ is maximal among all such sequences. 

In the case where such $\sigma_2$ and $\tau_2$ are found, we cannot have $\sigma_2 \subseteq U$ and $\tau_2 \subseteq V$, since then $(\sigma_2, \tau_2)$ would form a split pair, contradicting our assumption that $(\sigma_0, \tau_0) \nqvdash \Phi_{e_0}^{G_0 \oplus Y}(x_0) \downarrow \vee \Phi_{e_1}^{G_1 \oplus Y}(x_1) \downarrow$, and we can $Y'$-computably find whether $\sigma_2 \not \subseteq U$ or $\tau_2 \not \subseteq V$. Without any loss of generality, assume $\sigma_2 \not \subseteq U$. Define $\sigma_1 = \sigma_0$ and $\tau_1 = \tau_0 \cup \{\max_L \sigma_2\}$, then $(\sigma_1, \tau_1)$ is a condition extending $(\sigma_0, \tau_0)$. Indeed, since $\sigma_2 \not \subseteq U$, we have $\max_L \sigma_2 \in V$, and by construction $\max_L \sigma_2 <_L \min_L \tau_0$. Finally, we have $(\sigma_1, \tau_1) \Vdash_0 \Phi_{e_0}^{G_0 \oplus Y}(x_0) \uparrow$. Indeed, by minimality of $\max_L \sigma_2$, there is no $M$-finite $L$-ascending sequence $\sigma_2'$ such that $\sigma_1 = \sigma_0 \preceq \sigma_2'$, $\max_L \sigma_2' <_L \min_L \tau_1 = \max_L \sigma_2$ and $\Phi_{e_0}^{\sigma_2' \oplus Y}(x_0) \downarrow$.
\end{proof}

\begin{definition}
Fix a condition $(\sigma, \tau)$.

A \emph{$(\sigma, \tau)$-tree} is a binary tree $T$ labelled by a family of split pairs $(\sigma_{\rho}, \tau_{\rho})_{\rho \in T}$ such that for all $\rho \in T$:
\begin{itemize}
    \item $\sigma \preceq \sigma_{\rho}$ and $\tau \preceq \tau_{\rho}$ ;
    \item For every $\mu \in T$, if $\rho \cdot 0 \preceq \mu$ then $\sigma_{\rho} \preceq \sigma_{\mu}$ and if $\rho \cdot 1 \preceq \mu$ then $\tau_{\rho} \preceq \tau_{\mu}$ ;
\end{itemize}


     A node $\rho$ is said to be \emph{valid} if for every $\mu \in T$ such that $\mu \cdot 0 \preceq \rho$, $\sigma_{\mu} \subseteq U$ and for every $\mu \in T$ such that $\mu \cdot 1 \preceq \rho$, $\tau_{\mu} \subseteq V$. Being valid is a $\Delta^0_2(Y)$ condition.
\end{definition}

\begin{definition}
    Let $T$ a $(\sigma,\tau)$-tree 
    \begin{itemize}
        \item A node $(\sigma_{\rho}, \tau_{\rho})$ in $T$ \emph{satisfies} a pair of $\Sigma^0_1(Y)$-formulas $(\varphi_0(G_0, x), \varphi_1(G_1,x))$ if $\varphi_0(\sigma, |\rho| - \card \rho)$ and $\varphi_1(\tau, \card \rho)$ holds (where $\card \rho$ is the cardinal of $\rho$ seen as a set).
        \item $T$ \emph{satisfies} $(\varphi_0(G_0, x), \varphi_1(G_1,x))$ if every node of $T$ satisfies it.
    \end{itemize}
\end{definition}

\begin{lemma}\label[lemma]{lem:sads-wf-exists-maximal-tree}
For every pair $(\varphi_0(G_0, x), \varphi_1(G_1,x))$ of $\Sigma^0_1(Y)$ formulas and every condition $(\sigma, \tau)$, there is a maximal $\Sigma^0_1$ $(\sigma, \tau)$-tree~$T$ satisfying it.
\end{lemma}

\begin{proof}
$T$ will be the limit of the following increasing sequence of $Y$-computable $(\sigma, \tau)$-trees $(T_s)_{s \in M}$ that all satisfy $(\varphi_0(G_0, x), \varphi_1(G_1,x))$ (and therefore $T$ will also be a $(\sigma, \tau)$-tree satisfying $(\varphi_0(G_0, x), \varphi_1(G_1,x))$). 

At stage $0$, let $T_0$ be the empty tree. It vacuously satisfies $(\varphi_0(G_0, x), \varphi_1(G_1,x))$.

At stage $s+1$, if there exists a split pair $(\sigma', \tau')$ with $|\sigma'|,|\tau'| \leq s$, a $\rho \in T_s$ and an $i < 2$ with $\rho \cdot i \notin T_s$ such that adding $\rho \cdot i$ to $T_s$ with the index $(\sigma', \tau')$ still yields a $(\sigma, \tau)$-tree satisfying $(\varphi_0(G_0, x), \varphi_1(G_1,x))$, then let $T_{s+1}$ be that tree (choose the smallest split pair possible to enforce the maximality of $T$). If no such element is found, simply go to the next stage.
\end{proof}

\begin{lemma}\label[lemma]{lem:sads-wf-finite-tree-forcing}
Let $(\sigma,\tau)$ be a condition, $(\varphi_0(G_0, x), \varphi_1(G_1,x))$ be a pair of $\Sigma^0_1(Y)$ formulas, and $T$ be an $M$-finite maximal $(\sigma, \tau)$-tree satisfying $(\varphi_0(G_0, x), \varphi_1(G_1,x))$.
Then there is an extension $(\hat \sigma, \hat \tau) \leq (\sigma, \tau)$ and some~$p,q \in M \cup \{-1\}$ such that the following holds:
\begin{itemize}
    \item $(\hat \sigma, \hat \tau) \Vdash_0 \varphi_0(G_0, p)$ and $(\hat \sigma, \hat \tau) \Vdash_1 \varphi_1(G_1,q)$
    \item $(\hat \sigma, \hat \tau) \Vdash_0 \neg \varphi_0(G_1,p+1)$ or $(\hat \sigma, \hat \tau) \Vdash_1 \neg \varphi_1(G_1,q+1)$.
\end{itemize}
With the convention that  $(\hat \sigma, \hat \tau) \Vdash_i \varphi_i(G^i, -1)$ always holds.

\end{lemma}
\begin{proof}
If $T$ is empty, then $(\sigma,\tau) \nqvdash \varphi_0(G_0,0) \vee \varphi_1(G_1,0)$ as otherwise the corresponding split pair could be added in position $\epsilon$ in $T$, contradicting its maximality. Thus, by \Cref{lem:sads-forcing-question-sigma01}, there exists some extension $(\hat \sigma, \hat \tau) \leq (\sigma, \tau)$ and some side $i < 2$ such that $(\hat \sigma, \hat \tau) \Vdash_i \neg \varphi_i(G_i,0)$. Thus, we can assume that $T$ is not empty.

Consider a valid $\rho \in T$ having no valid descendant in $T$, this is possible since the set of valid element of $T$ is non empty because $(\sigma_\epsilon, \tau_\epsilon)$ is a vacuously valid, so we can take a maximal valid element of $T$ using $\mathsf{L}\Delta^0_2$ (which we have by $\BSig_2$).

Since $(\sigma_\rho, \tau_\rho)$ is a split pair, we have two cases:

Case 1: $\sigma_{\rho} \subseteq U$. In that case $\rho \cdot 0 \notin T$ by maximality of $\rho$. Then take $p = |\rho| - \card \rho$, $q = \card \rho - 1$ and $(\sigma', \tau') = (\sigma_{\rho}, \tau_{\rho \uh (i_0+1)})$ with $i_0$ the biggest index $i < |\rho|$ such that $\rho(i) = 1$ (if no such index exists we take $\tau$ instead of $\tau_{\rho \uh (i_0+1)}$), the pair $(\sigma', \tau')$ is a valid condition by definition of validity and the hypothesis $\sigma_{\rho} \subseteq U$. $(\sigma', \tau') \Vdash_0 \varphi_0(G_0, p)$ and $(\sigma', \tau') \Vdash_1 \varphi_1(G_1,q)$ because $T$ satisfies $(\varphi_0(G_0, x), \varphi_1(G_1,x))$. 

We have $(\sigma', \tau') \nqvdash \varphi_0(G_0,p + 1) \vee \varphi_1(G_1,q + 1)$ as otherwise the corresponding split pair could be added in position $\rho \cdot 0$ in $T$, contradicting its maximality. Hence, by \Cref{lem:sads-forcing-question-sigma01}, there exists some extension $(\hat \sigma, \hat \tau)$ and such that either $(\hat \sigma, \hat \tau) \Vdash_0 \neg \varphi_0(G_0,p + 1)$ or $(\hat \sigma, \hat \tau) \Vdash_1 \neg \varphi_1(G_1,q+1)$ holds.


Case 2: The case $\tau_{\rho} \subseteq V$ can be treated similarly. \\
\end{proof}

\begin{lemma}\label{lem:sads-wf-forcing-wf-jump}
Let $k \in M$, let $\Phi_k(G,x)$ be the formula: $\exists e_0 < e_1 < \dots < e_{x} < k$ such that $\Phi_{e_i}^{G \oplus Y}(e_i) \downarrow$ for all $i < x$, let $\Psi_k(G,x)$ be the formula: $\Gamma^{G \oplus Y}_k(x) \downarrow$ (with $\Gamma_k$ the functional of \Cref{lem:shore-blocking-wf} for the order $\alpha$).

For any combination $\phi_0(G,x), \phi_1(G,x) \in \{\Phi_k(G,x), \Psi_k(G,x)\}$ and every condition $(\sigma, \tau)$, there exists an extension  $(\hat \sigma, \hat \tau) \leq (\sigma, \tau)$ and some~$p,q \in M \cup \{-1\}$ such that the following holds:
\begin{itemize}
    \item $(\hat \sigma, \hat \tau) \Vdash_0 \varphi_0(G_0, p)$ and $(\hat \sigma, \hat \tau) \Vdash_1 \varphi_1(G_1,q)$
    \item $(\hat \sigma, \hat \tau) \Vdash_0 \neg \varphi_0(G_1,p+1)$ or $(\hat \sigma, \hat \tau) \Vdash_1 \neg \varphi_1(G_1,q+1)$.
\end{itemize}
With the convention that $(\hat \sigma, \hat \tau) \Vdash_i \varphi_i(G_i, -1)$ always holds.
\end{lemma}

\begin{proof}
By \Cref{lem:sads-wf-exists-maximal-tree} there exists a maximal $\Sigma^0_1$ $(\sigma, \tau)$-tree~$T$ satisfying $$(\varphi_0(G, x), \varphi_1(G,x)).$$ To apply \Cref{lem:sads-wf-finite-tree-forcing} and conclude, we just need to prove the $M$-finiteness of~$T$. 

By twisting $\Gamma_k$, we can assume that for every~$e, n \in M$, if $\Gamma^{\sigma \oplus Y}_k(n)\downarrow$, then
\begin{itemize}
    \item $\Gamma^{\sigma \oplus Y}_k(m)\downarrow$ for every~$m < n$ ;
    \item $\Gamma^{\sigma \oplus Y}_k(0), \Gamma^{\sigma \oplus Y}_k(1), \dots, \Gamma^{\sigma \oplus Y}_k(m)$ is a strictly $<_{\epsilon_0}$-decreasing sequence of elements smaller than~$\beta \ntimes k$.
\end{itemize}

Assume by contradiction that $T$ is $M$-infinite and let $(T_s)_{s \in M}$ be an increasing sequence of computable tree approaching $T$. Without loss of generality we can assume that $T_s \neq T_{s+1}$ for all $s \in M$.

For every $s \in M$ and $\rho$ node of $T_s$ indexed by $(\sigma_{\rho}, \tau_{\rho})$ with set of immediate successors $S_\rho$ ($\card S_\rho \leq 2$ since $T_s$ is a binary tree), define inductively :
$$\zeta_s(\rho) = \omega^{f_0(\rho) \nplus f_1(\rho)}(2 - \card S_\rho)\ \nplus \nsum_{\mu \in S_\rho} \zeta_s(\mu)
$$
Where $f_0(\rho) = \Gamma_k^{\sigma_{\rho} \oplus Y}(|\rho| - \card \rho)$ if $\varphi_0(G,x) = \Psi_k(G,x)$ and $f_0(\rho) = k - (|\rho| - \card \rho)$ if $\varphi_0(G,x) = \Phi_k(G,x)$. And $f_1(\rho) = \Gamma_k^{\tau_{\rho} \oplus Y}(\card \rho)$ if $\varphi_1(G,x) = \Psi_k(G,x)$ and $f_1(\rho) = k - \card \rho$ if $\varphi_1(G,x) = \Phi_k(G,x)$.

And let $\zeta_s = \zeta_s(\epsilon)$. \\

\textbf{Claim}. \emph{The sequence $\zeta_s$ is an $M$-infinitely $<_{\epsilon_0}$-decreasing sequence of ordinals less than $\omega^{\beta \times \omega}$}

Let $s \in M$, since $T_{s+1} \neq T_s$, there is a $\rho \in T_s$ an $i < 2$ such that $T_{s+1} = T_s \cup \rho \cdot i$. Therefore:

$$\zeta_{s}(\rho) \nplus (\omega^{f_0(\rho \cdot i) + f_1(\rho \cdot i)} \ntimes 2) = \zeta_{s+1}(\rho) \nplus \omega^{f_0(\rho) \nplus f_1(\rho)}$$

But for every possible values of $\varphi_0$ and $\varphi_1$, we have $f_0(\rho) > f_0(\rho \cdot 0)$, $f_0(\rho) = f_0(\rho \cdot 1)$, $f_1(\rho) = f_1(\rho \cdot 0)$ and $f_1(\rho) > f_1(\rho \cdot 1)$. So for every possible value of $i$, we have 
$$\omega^{f_0(\rho) + f_1(\rho)} > \omega^{f_0(\rho \cdot i) + f_1(\rho \cdot i)} \times 2$$
and therefore $\zeta_{s+1}(\rho) < \zeta_s(\rho)$ which gives us $\zeta_{s+1} < \zeta_s$ (as $\rho \cdot i$ is the only difference between $T_s$ and $T_{s+1}$). \\

Therefore, for all $s \in M$, $\zeta_s \leq_{\epsilon_0} \zeta_0 \leq_{\epsilon_0} \omega^{\beta \times \omega}$. \\

This contradicts $\WF(\sup_{\beta < \alpha} \omega^{\beta \times \omega})$, so $T$ is $M$-finite and the result follows.
\end{proof}

\begin{remark}
If $(\sigma, \tau) \Vdash_i  \Phi_k(G_i,p-1)$ and $(\sigma, \tau) \Vdash_i \neg \Phi_k(G_i,p)$, then there exists a $\rho' \in 2^k$ such that  $(\sigma, \tau) \Vdash_i (G_i \oplus Y)' \uh k = \rho'$. Such a $\rho'$ can be found computably in $Y$. 

Moreover, if $(\sigma, \tau) \Vdash_i  \Psi_k(G_i,p-1)$ and $(\sigma, \tau) \Vdash_i \neg \Psi_k(G_i,p)$ then $(\sigma, \tau) \Vdash_i \Gamma_{s_0}^{G_i \oplus Y}(p) \uparrow$, so $(\sigma, \tau)$ forces $\Gamma_{s_0}^{G_i \oplus Y}$ to be partial. 
\end{remark}

\begin{remark}
Because of the disjunctive requirements, preservation of~$\WF(\epsilon_0)$ and first jump control must be combined. Thus, at first sight, $\WF(\epsilon_0)$ is used to control the first jump. However, if one wants to control only the first jump, no well-foundedness assumption is necessary. Indeed, let $U$ be the set of pairs $(p,q) \in k^2$ such that there is a pair $(\sigma_1, \tau_1)$ with $\sigma_1 \succeq \sigma$, $\tau_1 \succeq \tau$ and $\max_\L \sigma_1 < \min_\L \tau_1$ satisfying $\Psi_k(\sigma_1, p)$ and $\Psi_k(\tau_1, q)$. Then $U$ exists by bounded $\Sigma^0_1$ comprehension. Using this set~$U$, one can find an extension with the desired property. Therefore, one can prove that $\RCA_0 + \BSig_2 + \SADS$ is a $\Pi^1_1$-conservative extension of~$\RCA_0 + \BSig_2$, as already proved by Chong, Slaman and Yang~\cite{chong2021pi11} using the first construction.
\end{remark}
\smallskip

\textbf{Construction} We will build a decreasing sequence $(\sigma_s,\tau_s)$ of conditions and let $G_0$ and $G_1$ to be the union of the $\sigma_s$ and $\tau_s$. We will also build two increasing sequence $(\sigma_s')$ and $(\tau_s')$ such that $G_0'$ and $G_1'$ will be the union of the $\sigma_s'$ and $\tau_s'$. Initially we take $\sigma_0 = \sigma_0' = \tau_s = \tau_s' = \epsilon$, and during the construction we will ensure that we have $|\sigma_s|,|\sigma'_s|,|\tau_s|,|\tau_s'| \leq s$. 



Each stage of the construction will be of one of the following 5 types corresponding to the requirements: $\R \vee \R$, $\S \vee \R$, $\R \vee \S$, $\S \vee \S$ or of type $\mathcal{T} \wedge \mathcal{T}$. The stage $0$ is of type $\R \vee \R$.

Assume that $(\sigma_s, \tau_s)$, $\sigma'_s$ and $\tau_s'$ are already defined. Let $s_0 < s$ be the latest stage at which we switched the stage type. We have 5 cases.


Case 1: $s$ is of type $\R \vee \R$. If there exists some condition $(\sigma,\tau) \leq (\sigma_s,\tau_s)$ with $|\sigma|,|\tau| \leq s$ such that $(\sigma,\tau) \Vdash_0 (G_0 \oplus Y)' \uh_{s_0} = \sigma'$ for a $\sigma' \in 2^{s_0}$ or $(\sigma,\tau) \Vdash_1 (G_1 \oplus Y)' \uh_{s_0} = \tau'$ for a $\tau' \in 2^{s_0}$. Then let $\sigma_{s+1} = \sigma$, $\tau_{s+1} = \tau_s$, and if we are in the first case, let $\sigma'_{s+1} = \sigma'$ and $\tau'_{s+1} = \tau'_s$ and in the second case let $\sigma'_{s+1} = \sigma'_s$ and $\tau'_{s+1} = \tau'$. Then let $s+1$ be of type $\R \vee \S$. Otherwise, the elements are left unchanged and we go to the next stage.

Case 2: $s$ is of type $\R \vee \S$. If there exists some condition $(\sigma,\tau) \leq (\sigma_s,\tau_s)$ with $|\sigma|,|\tau| \leq s$ such that $(\sigma,\tau) \Vdash_0 (G_0 \oplus Y)' \uh_{s_0} = \sigma'$ for a $\sigma' \in 2^{s_0}$ or $(\sigma,\tau) \Vdash_1 \Gamma_{s_0}^{G_1 \oplus Y} partial$. Then let $\sigma_{s+1} = \sigma$, $\tau_{s+1} = \tau_s$, and if we are in the first case, let $\sigma'_{s+1} = \sigma'$ and $\tau'_{s+1} = \tau'_s$ and in the second case let $\sigma'_{s+1} = \sigma'_s$ and $\tau'_{s+1} = \tau'_s$. Then let $s+1$ be of type $\S \vee \R$. Otherwise, the elements are left unchanged and we go to the next stage.

Case 3: $s$ is of type $\S \vee \R$. If there exists some condition $(\sigma,\tau) \leq (\sigma_s,\tau_s)$ with $|\sigma|,|\tau| \leq s$ such that $(\sigma,\tau) \Vdash_0\Gamma_{s_0}^{G_0 \oplus Y} partial$ or $(\sigma,\tau) \Vdash_1 (G_1 \oplus Y)' \uh_{s_0} = \tau'$ for a $\tau' \in 2^{s_0}$. Then let $\sigma_{s+1} = \sigma$, $\tau_{s+1} = \tau_s$, and if we are in the first case, let $\sigma'_{s+1} = \sigma'_s$ and $\tau'_{s+1} = \tau'_s$ and in the second case let $\sigma'_{s+1} = \sigma'_s$ and $\tau'_{s+1} = \tau'$. Then let $s+1$ be of type $\S \vee \S$. Otherwise, the elements are left unchanged and we go to the next stage.

Case 4: $s$ is of type $\S \vee \S$. If there exists some condition $(\sigma,\tau) \leq (\sigma_s,\tau_s)$ with $|\sigma|,|\tau| \leq s$ such that $(\sigma,\tau) \Vdash_0\Gamma_{s_0}^{G_0 \oplus Y} partial$ or $(\sigma,\tau) \Vdash_1 \Gamma_{s_0}^{G_1 \oplus Y} partial$. Then let $\sigma_{s+1} = \sigma$, $\tau_{s+1} = \tau_s$, and if we are in the first case, let $\sigma'_{s+1} = \sigma'_s$ and $\tau'_{s+1} = \tau'_s$ and in the second case let $\sigma'_{s+1} = \sigma'_s$ and $\tau'_{s+1} = \tau_s'$. Then let $s+1$ be of type $\mathcal{T}$. Otherwise, the elements are left unchanged and we go to the next stage.

Case 5: $s$ is of type $\mathcal{T}$. If there exists some condition $(\sigma, \tau) \leq (\sigma_s, \tau_s)$ with $\card \sigma > \card \sigma_s$ and $\card \tau > \card \tau_s$, then let $\sigma_{s+1} = \sigma$, $\tau_{s+1} = \tau$, $\sigma'_{s+1} = \sigma'_s$ and $\tau'_{s+1} = \tau'_s$. Then let $s+1$ be of type $\R \vee \R$. Otherwise, the elements are left unchanged and we go to the next stage. \\ 

This completes the construction.
\bigskip

\textbf{Verification}. Since the size of $\sigma_s, \tau_s, \sigma'_s$ and $\tau_s$ are bounded by $s$, there is a $\Delta_2^0(Y)$-formula $\phi(s)$ stating that the construction can be pursued up to stage $s$. Our construction implies that the set $\{s|\phi(s)\}$ is a cut, so by $\IDel_2(Y)$, the construction can be pursued at every stage.

Let~$G_0 = \bigcup_{s \in M} \sigma_s$ and $G_1 = \bigcup_{s \in M} \tau_s$.
By \Cref{lem:sads-wf-forcing-wf-jump}, and the fact that $U$ and $V$ are $M$-infinite, each type of stage changes $M$-infinitely often. Thus, for $M$-infinitely many~$k$'s, the four requirements $\R^0_k \vee \R^1_k$, $\S^0_k \vee \R^1_k$, $\R^0_k \vee \S^1_k$ and $\S^0_k \vee \S^1_k$ are satisfied and $G_0$ and $G_1$ are $M$-infinite. 

So, for $M$-infinitely many~$k$'s, $(\mathcal{R}^0_k \wedge \S^0_k) \vee (\R^1_k \wedge \S^1_k)$ is satisfied which gives us an $i < 2$ such that $G_i$ is an ascending or descending sequence, $Y' \geq_T (G_i \oplus Y)'$ and therefore $\M[G_i] \models \BSig_2$ and $\M[G_i] \models \WF(\alpha)$. This completes the proof of \Cref{prop:sads-wf}.
\end{proof}

\begin{theorem}\label[theorem]{thm:sads-pi11-wfepsilon0}
$\RCA_0 + \BSig_2 + \WF(\epsilon_0) + \SADS$ is $\Pi_1^1$-conservative over $\RCA_0 + \BSig_2 + \WF(\epsilon_0)$.
\end{theorem}

\begin{proof}
Immediate by \Cref{prop:sads-wf} for the case $\alpha = \epsilon_0$ and \Cref{theorem:problem-wf-cons}.
\end{proof}

\section{$\Pi_1^1$-conservation of $\RT_2^2$}\label[section]{sect:conservation-rt22}

Thanks to the decomposition of \Cref{prop:rt22-sem-sads-coh}, we can combine the conservation theorems developed in \Cref{sect:bsig-wfalpha} and \Cref{sect:bsig2-epsilon0} to prove our main theorems. Yokoyama~\cite{yokoyama2010conservativity} proved the following amalgamation theorem:

\begin{theorem}[\cite{yokoyama2010conservativity}]
Let~$T_0, T_1, T_2$ be $\Pi^1_2$ theories such that $T_0 \supseteq \RCA_0$ and $T_1$ and $T_2$ are both $\Pi^1_1$-conservative extensions of~$T_0$. Then $T_1 + T_2$ is a $\Pi^1_1$-conservative extension of~$T_0$.
\end{theorem}

In particular, letting $T_0$ be $\RCA_0 + \BSig_2 + \WF(\epsilon_0)$, one can combine \Cref{thm:coh-pi11-wfepsilon0}, \Cref{thm:wkl-pi11-wfepsilon0}, \Cref{thm:sem-pi11-wfepsilon0} and \Cref{thm:sads-pi11-wfepsilon0} to deduce that $\RCA_0 + \RT^2_2 + \WKL + \WF(\epsilon_0)$ is a $\Pi^1_1$-conservative extension of~$\RCA_0 + \BSig_2 + \WF(\epsilon_0)$. We shall however provide a direct proof using $\omega$-extensions with the following proposition.

\begin{proposition}\label[proposition]{prop:rt22-wf}
Consider a countable model $\M = (M,S) \models \RCA_0 + \BSig_2 + \WF(\omega_4^{\alpha})$ topped by a set~$Y \in S$, where~$\alpha \in M$ is an ordinal~$\leq \epsilon_0$. Then, for every coloring~$f : [M]^2 \to 2$ in~$S$ and every set~$P \gg Y'$ such that $\M[P] \models \RCA_0^*$, there exists $G \subseteq M$ such that 
\begin{enumerate}
    \item $G$ is an $M$-infinite $f$-homogeneous set ;
    \item $P \geq_T (G \oplus Y)'$ ;
    \item $\M[G] \models \RCA_0 + \BSig_2 + \WF(\alpha)$.
\end{enumerate}
\end{proposition}
\begin{proof}
By Fiori-Carones et al~\cite[Lemma 3.2]{fiori2021isomorphism}, there exists a $P$-computable countable coded model of~$\WKL_0^*$ containing~$Y'$. In particular, there exists a set~$P_0 \subseteq M$ such that $P \gg P_0 \gg Y'$.
Let~$\vec{R} = R_0, R_1, \dots$ be defined by $R_x = \{ y : f(x, y) = 1 \}$.
By \Cref{prop:coh-wf} applied to~$\M$ and $P_0$, there exists some $M$-infinite $\vec{R}$-cohesive set~$G_0 \subseteq M$ such that $P_0 \geq_T (G_0 \oplus Y)'$ and $\M[G_0] \models \RCA_0 + \BSig_2 + \WF(\omega_3^{\alpha})$. In particular, $f : [G_0]^2 \to 2$ is stable. One can see~$f : [G_0]^2 \to 2$ as an instance of~$\SEM$.
By \Cref{prop:sem-wf} applied to~$\M[G_0]$ and $P$, there exists some~$M$-infinite $f$-transitive set~$G_1 \subseteq G_0$ such that $P \geq_T (G_1 \oplus G_0 \oplus Y)'$ and $\M[G_1 \oplus G_0] \models \RCA_0 + \BSig_2 + \WF(\omega_2^\alpha)$. One can see $f : [G_1]^2 \to 2$ as an instance of~$\SADS$.
By \Cref{prop:sads-wf} applied to~$\M[G_0 \oplus G_1]$ and $P$, there exists some~$M$-infinite $f$-homogeneous set~$G \subseteq G_1$ such that $P \geq_T (G \oplus G_1 \oplus G_0 \oplus Y)'$ and $\M[G \oplus G_1 \oplus G_0] \models \RCA_0 + \BSig_2 + \WF(\alpha)$.
\end{proof}


\begin{corollary}\label[corollary]{cor:rt22-wf}
Fix~$n \geq 2$. Let~$T_n$ be either~$\BSig_n$ or~$\ISig_n$.
Consider a countable model $\M = (M,S) \models \RCA_0 + T_n + \WF(\omega_4^{\alpha})$ topped by a set~$Y \in S$, where~$\alpha \in M$ is an ordinal~$\leq \epsilon_0$. Then, for every coloring~$f : [M]^2 \to 2$ in~$S$, there exists $G \subseteq M$ such that 
\begin{enumerate}
    \item $G$ is an $M$-infinite $f$-homogeneous set ;
    \item $\M[G] \models \RCA_0 + T_n + \WF(\alpha)$.
\end{enumerate}
\end{corollary}
\begin{proof}
By Harrington; Simpson–Smith; Hajek (see Belanger~\cite{Blanger2022ConservationTF}), every countable topped model of $\RCA_0^* + T_{n-1}$ can be $\omega$-extended into a model of~$\RCA_0^* + T_{n-1} + \WKL$. Since~$\M[Y'] \models \RCA_0^* + T_{n-1}$, there is some $P \subseteq M$ such that $P \gg Y'$ and $\M[P] \models \RCA_0^* + T_{n-1}$. By \Cref{prop:rt22-wf}, there exists $G \subseteq M$ such that 
\begin{enumerate}
    \item $G$ is an $M$-infinite $f$-homogeneous set ;
    \item $P \geq_T (G \oplus Y)'$ ;
    \item $\M[G] \models \RCA_0 + \BSig_2 + \WF(\alpha)$.
\end{enumerate}
Since every $\Sigma^0_n$ formula in~$\M[G]$ is $\Sigma^0_{n-1}$ in~$\M[P]$
and $\M[P] \models T_{n-1}$, $\M[G] \models T_n$.
\end{proof}



\begin{repmaintheorem}{main:theorem1}
Fix~$n \geq 2$. Let~$T_n$ be either~$\BSig_n$ or~$\ISig_n$.
$\RCA_0 + T_n + \RT^2_2 + \WKL + \WF(\epsilon_0)$ is a $\Pi^1_1$-conservative extension of $\RCA_0 + T_n + \WF(\epsilon_0)$.
\end{repmaintheorem}

\begin{proof}
From \Cref{cor:rt22-wf}, \Cref{prop:wkl-wf} and \Cref{theorem:problem-wf-cons}, we get the $\Pi_1^1$-conservativity of $\WKL + \RT^2_2$ over $\BSig_2 + \WF(\epsilon_0)$.
\end{proof}

\subsection{Preserving $\bigcup_n \WF(\omega_n)$}\label[section]{sect:bigcup-wf-omegan}

In their proof that $\RT^2_2$ does not imply~$\ISig_2$ over~$\RCA_0$, Chong, Slaman and Yang~\cite{chong2017inductive} started with a model satisfying a slightly weaker well-foundedness axiom than $\WF(\epsilon_0)$, namely, $\bigcup_n \WF(\omega_n)$. We show that our assumption can also be weakened to $\bigcup_k \WF(\omega^\omega_k)$ with a little bit of extra work.

\begin{theorem}\label[theorem]{theorem:problem-wf-cons-2}
Fix~$n \geq 2$. Let~$T_n$ be either~$\BSig_n$ or~$\ISig_n$.
Let $\mathsf{P}$ be a $\Pi_2^1$ problem such that for every countable topped model $\M = (M,S) \models \RCA_0 + T_n + \WF(\omega_k^{\alpha})$ with $\alpha \in M$ an ordinal less than $\epsilon_0$ and $k$ a standard integer and every instance of $\mathsf{P}$ in $\M$ there exists a solution $G \subset M$ of that instance such that $\M[G] \models \RCA_0 + T_n + \WF(\alpha)$. Then $\mathsf{P}$ is $\Pi_1^1$-conservative over $\RCA_0 + T_n + \bigcup_{k \in \omega} \WF(\omega^\omega_k)$.   
\end{theorem}

\begin{proof}
Assume $\RCA_0 + T_n + \bigcup_{n \in \omega} \WF(\omega_n) \not \vdash \forall X \phi(X)$ for $\phi$ arithmetical.  

Extend the language of second order arithmetics with a first-order constant symbol $c$. The theory $\RCA_0 + T_n + \WF(\omega_c) + \bigcup_{k \in \omega} c > k + \exists X \neg \phi(X)$ is finitely satisfiable by our assumption, so by compactness we have a model $\M \models \RCA_0 + T_n + \WF(\omega_c) + \bigcup_{k \in \omega} c > k + \exists X \neg \phi(X)$. Such a $c$ will therefore be non-standard.

From our assumption, for any countable topped model $\mathcal{N} = (N,S) \models \RCA_0 + T_n + \WF(\omega_{d+k})$ and instance of $\mathsf{P}$ in this model, we can find a countable topped $\omega$-extension that will be a model of $\RCA_0 + T_n + \WF(\omega_{d})$ and will contain a solution to that instance.

So by iterating that reasoning, we can build a sequence  $\M_0 = \M \subseteq \M_1 \subseteq \dots$ of countable topped model such that $\M_s \models \RCA_0 + T_n + \WF(\omega_{c - ks})$ and such that every instance of $P$ appearing inside one $\M_i$ will eventually have a solution in one $\M_j$. Since $c$ is non-standard and every such $s$ is standard, every $\M_i$ is a model of $\bigcup_{k \in \omega} \WF(\omega^\omega_k)$, and the model $\bigcup_{s \in \omega} \M_s$ will therefore be a model of $\RCA_0 + T_n + \bigcup_{k \in \omega} \WF(\omega^\omega_k) + \exists X \neg \phi(X) + \mathsf{P}$.

So $\RCA_0 + T_n + \bigcup_{k \in \omega} \WF(\omega^\omega_k)  + \mathsf{P} \not \vdash \forall X \phi(X)$.
\end{proof}

\begin{repmaintheorem}{main:theorem2}
Fix~$n \geq 2$. Let~$T_n$ be either~$\BSig_n$ or~$\ISig_n$.
$\RCA_0 + T_n + \bigcup_{k \in \omega} \WF(\omega^\omega_k) + \RT_2^2 + \WKL$ is $\Pi_1^1$-conservative over $\RCA_0 + T_n + \bigcup_{k \in \omega} \WF(\omega^\omega_k)$.   
\end{repmaintheorem}

\begin{proof}
From \Cref{cor:rt22-wf}, \Cref{prop:wkl-wf} and \Cref{theorem:problem-wf-cons-2}, we get the $\Pi_1^1$-conservativity of $\RCA_0 + T_n + \RT^2_2 + \WKL$ over $\RCA_0 + T_n + \bigcup_{k \in \omega} \WF(\omega^\omega_k)$. 
\end{proof}

\begin{remark}
The combination of \Cref{prop:coh-wf}, \Cref{prop:sem-wf} and  \Cref{prop:sads-wf} provides the $\ll_2$-basis theorem for $\RT^2_2$ while preserving $\WF(\epsilon_0)$ (see \cite[Definition 4.9]{fiori2021isomorphism}).
This implies that there exists a polynomial proof transformation for the $\Pi^1_1$-conservation for \Cref{main:theorem1} (it will appear in the forthcoming thesis of H. Ikari).
On the other hand, it is not clear whether there exists a polynomial proof transformation for \Cref{main:theorem2}.
\end{remark}

\subsection{Preserving $\ISig_n$ and $\BSig_{n+1}$ for~$n \geq 2$}

One can use the previous theorems to reprove a known conservation result over~$\ISig_2$.

\begin{proposition}\label[proposition]{prop:rt22-sig2}
Fix~$n \geq 2$ and let $T_n$ be either $\ISig_n$, or $\BSig_{n+1}$.
Consider a countable model $\M = (M,S) \models \RCA_0 + T_{n-1}$ topped by a set~$Y \in S$. Then, for every coloring~$f : [M]^2 \to 2$ in~$S$ and every set~$P \gg Y'$ such that $\M[P] \models \RCA_0$, there exists $G \subseteq M$ such that 
\begin{enumerate}
    \item $G$ is an $M$-infinite $f$-homogeneous set ;
    \item $P \geq_T (G \oplus Y)'$ ;
    \item $\M[G] \models \RCA_0 + T_n$.
\end{enumerate}
\end{proposition}
\begin{proof}
By Fiori-Carones et al~\cite[Lemma 3.2]{fiori2021isomorphism}, there exists a $P$-computable countable coded model of~$\WKL_0^*$ containing~$Y'$.
In particular, there exists a set~$P_0 \subseteq M$ such that $P \gg P_0 \gg Y'$. Moreover, $\M[P]$ and $\M[P_0]$ both satisfy~$T_{n-1}$.

Let~$\vec{R} = R_0, R_1, \dots$ be defined by $R_x = \{ y : f(x, y) = 1 \}$.
By \Cref{prop:coh-wf} applied to~$\M$, $P_0$ and $\alpha = 1$, there exists some $M$-infinite $\vec{R}$-cohesive set~$G_0 \subseteq M$ such that $P_0 \geq_T (G_0 \oplus Y)'$ and $\M[G_0] \models \RCA_0 + \BSig_2$.  In particular, $f : [G_0]^2 \to 2$ is stable.
Since any~$\Sigma^0_n(G_0 \oplus Y)$ formula is $\Sigma^0_{n-1}(P_0)$, and $\M[P_0] \models T_{n-1}$, $\M[G_0] \models T_n$. In particular, since~$\RCA_0 + T_n \vdash \WF(\omega^\omega$, $\M[G_0] \models \WF(\omega^\omega)$.

One can see~$f : [G_0]^2 \to 2$ as an instance of~$\SEM$.
By \Cref{prop:sem-wf} applied to~$\M[G_0]$, $P$ and $\alpha = 1$, there exists some~$M$-infinite $f$-transitive set~$G_1 \subseteq G_0$ such that $P \geq_T (G_1 \oplus G_0 \oplus Y)'$.
Here again, since any~$\Sigma^0_n(G_1 \oplus G_0 \oplus Y)$ formula is $\Sigma^0_{n-1}(P)$ and $\M[P] \models T_{n-1}$, $\M[G_1 \oplus G_0] \models T_n$. 

One can see $f : [G_1]^2 \to 2$ as an instance of~$\SADS$.
By \Cref{prop:sads-wf} applied to~$\M[G_0 \oplus G_1]$, $P$ and $\alpha = 0$, there exists some~$M$-infinite $f$-homogeneous set~$G \subseteq G_1$ such that $P \geq_T (G \oplus G_1 \oplus G_0 \oplus Y)'$ and $\M[G \oplus G_1 \oplus G_0] \models \RCA_0 + \BSig_2$.
Still by the same argument, $\M[G \oplus G_1 \oplus G_0] \models T_n$.
\end{proof}

\begin{remark}
Contrary to \Cref{prop:coh-wf} and \Cref{prop:sads-wf} in which the well-foundedness hypothesis on the ground model is only used to preserve well-foundedness, in the proof of \Cref{prop:sem-wf}, one exploits $\WF(\omega^\omega)$ to preserve~$\BSig_2$, and one uses $\WF(\omega^{\omega^\omega \times \omega})$ to preserve~$\WF(\omega^\omega)$.
If one starts with a model of~$\ISig_2$, the ground model is therefore also a model of~$\WF(\omega^\omega)$, in which case $\BSig_2$ can be preserved. Thanks to the first-jump control, $\ISig_2$ is actually preserved, and therefore so is $\WF(\omega^\omega)$. 

Since $\ISig_2$ and $\WF(\omega^{\omega^\omega \times \omega})$ are incomparable over~$\RCA_0$, it would be interesting to investigate the common factor between these two statements which is sufficient to preserve~$\WF(\omega^\omega)$.
\end{remark}

\begin{corollary}\label[corollary]{cor:rt22-sig2}
Fix~$n \geq 2$ and let $T_n$ be either $\ISig_n$, or $\BSig_{n+1}$.
Consider a countable model $\M = (M,S) \models \RCA_0 + T_n$ topped by a set~$Y \in S$. Then, for every coloring~$f : [M]^2 \to 2$ in~$S$, there exists $G \subseteq M$ such that 
\begin{enumerate}
    \item $G$ is an $M$-infinite $f$-homogeneous set ;
    \item $\M[G] \models \RCA_0 + T_n$.
\end{enumerate}
\end{corollary}
\begin{proof}
By Harrington; Simpson–Smith; Hajek (see Belanger~\cite{Blanger2022ConservationTF}), every countable topped model of $\RCA_0^* + T_{n-1}$ can be $\omega$-extended into a model of~$\RCA_0^* + T_{n-1} + \WKL$. Since~$\M[Y'] \models \RCA_0^* + T_{n-1}$, there is some $P \subseteq M$ such that $P \gg Y'$ and $\M[P] \models \RCA_0^* + T_{n-1}$.
The conclusion follows from~\Cref{prop:rt22-sig2}.    
\end{proof}

\begin{corollary}[Cholak, Jockusch and Slaman~\cite{cholak_jockusch_slaman_2001}]
Fix~$n \geq 2$ and let $T_n$ be either $\ISig_n$, or $\BSig_{n+1}$.
$\RCA_0 + T_n + \RT^2_2$ is $\Pi^1_1$-conservative over~$\RCA_0 + T_n$.
\end{corollary}

\section{Open questions}\label[section]{sect:open-questions}

Many open questions remain around the first-order part of Ramsey's theorem for pairs and its consequences. The first and most important question being the following:

\begin{question}
Is $\RCA_0 + \RT^2_2$ $\Pi^1_1$-conservative over~$\RCA_0 + \BSig_2$?
\end{question}

For every $\alpha < \epsilon_0$, $\RCA_0 + \BSig_2 + \WF(\alpha) + \WKL + \COH$ is~$\Pi^1_1$
conservative over~$\RCA_0 + \BSig_2 + \WF(\alpha)$. The situation seems more complex for~$\ADS$ and~$\EM$, 
which both seem to require $\WF(\sup_{\beta < \alpha} \omega^{\beta \times \omega})$ to preserve $\WF(\alpha)$.

\begin{question}
For which~$\alpha < \epsilon_0$ is $\RCA_0 + \WF(\alpha) + \RT^2_2$ $\Pi^1_1$-conservative over~$\RCA_0 + \BSig_2 + \WF(\alpha)$?
\end{question}

The proofs of \Cref{prop:sem-wf} and \Cref{prop:sads-wf} required the ground model to satisfy~$\WF(\omega^\alpha)$ to preserve~$\WF(\alpha)$.
It is natural to wonder whether this assumption is necessary. One way to formalize this idea would be to prove that $\RCA_0 + \RT^2_2 \rightarrow \forall n (\WF(\omega_n) \rightarrow  \WF(\omega_{n+1}))$, where $\omega_0 = 1$ and $\omega_{n+1} = \omega^{\omega_n}$. This is not the case, as $\RCA_0 + \RT^2_2$ is $\forall \Pi^0_3$-conservative over~$\RCA_0$, hence does not prove $\WF(\omega^\omega)$. One could also wonder whether $\RCA_0 + \RT^2_2$ proves some weaker well-foundedness closure properties which do not already follow from~$\RCA_0 + \BSig_2$. Here again, the answer is negative. Le Hou\'erou and Levy Patey constructed a model $\M = (M, S)$ of $\RCA_0 + \RT^2_2$ in which $\{ b \in M : \M \models \WF(\omega^b) \} = \sup \{ a \cdot k : k \in \omega \}$ for some~$a \in M \setminus \omega$. In other words, this cut has no better closure properties than additivity [personal communication]. This result is tight, as $\RCA_0 \vdash \forall a(\WF(\omega^a) \rightarrow \WF(\omega^{2a}))$. Note that one can also show that $\RCA_0 + \RT^2_2$ does not prove $\forall n (\WF(\omega^n) \rightarrow \WF(\omega^{2^n}))$ using proof length. Indeed, by Ko{\l}odziejczyk, Wong and Yokoyama~\cite{kolodziejczyk2023ramsey}, $\RCA_0 + \RT^2_2$ does not have speedup over~$\RCA_0$ for~$\forall \Pi^0_3$ consequences, but in an incoming paper, Ko{\l}odziejczyk and Yokoyama [personal communication] proved that $\RCA_0 + \forall n (\WF(\omega^n) \rightarrow \WF(\omega^{2^n}))$ has super-exponential speedup over~$\RCA_0$ for~$\forall \Pi^0_3$ consequences.



\bigskip
\begin{center}
\textbf{Acknowledgements}
\end{center}
The authors are thankful to Leszek Ko{\l}odziejczyk and the anonymous reviewer for their careful reading and improvement suggestions.
Yokoyama is partially supported by JSPS KAKENHI grant numbers JP19K03601, JP21KK0045 and JP23K03193.

\bibliographystyle{plain}
\bibliography{biblio}

\end{document}